\documentclass[preprint,12pt]{elsarticle}
\usepackage{hyperref}
\usepackage{amsmath,amssymb,tikz}
\usetikzlibrary{arrows,shapes,positioning,decorations.markings}
\newtheorem{theorem}{Theorem}[section]
\newtheorem{lemma}[theorem]{Lemma}
\newtheorem{corollary}[theorem]{Corollary}

\newtheorem{proposition}[theorem]{Proposition}

\newcommand{\Haar}{\mathrm{Haar}}
\newcommand{\vGa}{{\Gamma}}

\newcommand{\Aut}{\mathop{\mathrm{Aut}}}

\numberwithin{equation}{section}

\newproof{proof}{Proof}

\makeatletter
\newcommand{\myitem}[1]{%
\item[#1]\protected@edef\@currentlabel{#1}%
}
\makeatother

\begin{document}

\begin{frontmatter}
\title{Haar graphical representations of finite groups and an application to poset representations} 

\author{Joy Morris\fnref{fn1}}           
\affiliation{organization={Department of Mathematics and Computer Science},
addressline={University of Lethbridge},
city={Lethbridge},
postcode={T1K 3M4},
state={AB},
country={Canada}}
\ead{joy.morris@uleth.ca}
\fntext[fn1]{Supported by the Natural Science and Engineering Research Council of Canada (grant RGPIN-2017-04905).}

\author{Pablo Spiga}
\affiliation{
organization={Dipartimento di Matematica e Applicazioni, University of Milano-Bicocca},
addressline={Via Cozzi 55},
city={Milano},
postcode={20125},
country={Italy}}
\ead{pablo.spiga@unimib.it}


\begin{abstract}
Let $R$ be a group and let $S$ be a subset of $R$. The Haar graph $\mathrm{Haar}(R,S)$ of $R$ with connection set $S$ is the graph having vertex set $R\times\{-1,1\}$, where two distinct vertices $(x,-1)$ and $(y,1)$ are declared to be adjacent if and only if $yx^{-1}\in S$. The name Haar graph was coined by Toma\v{z} Pisanski in one of the first investigations on this class of graphs.

For every $g\in R$, the mapping $\rho_g:(x,\varepsilon)\mapsto (xg,\varepsilon)$, $\forall (x,\varepsilon)\in R\times\{-1,1\}$, is an automorphism of $\mathrm{Haar}(R,S)$. In particular, the set $\hat{R}:=\{\rho_g\mid g\in R\}$ is a subgroup of the automorphism group of $\mathrm{Haar}(R,S)$ isomorphic to $R$. In the case that the automorphism group of $\mathrm{Haar}(R,S)$ equals $\hat{R}$, the Haar graph $\mathrm{Haar}(R,S)$ is said to be a Haar graphical representation of the group $R$.

Answering a question of Feng, Kov\'{a}cs, Wang, and Yang, we classify the finite groups admitting a Haar graphical representation. Specifically, we show that every finite group admits a Haar graphical representation, with abelian groups and ten other small groups as the only exceptions.

Our work on Haar graphs allows us to improve a 1980 result of Babai concerning representations of groups on posets, achieving the best possible result in this direction. An improvement to Babai's related result on representations of groups on distributive lattices follows. 
\end{abstract}

\end{frontmatter}
\section{Introduction}\label{intro}

\subsection{Background}

All digraphs and groups considered in this paper are finite. A \textbf{\em digraph} $\vGa$ is an ordered pair $(V,A)$ where the \textbf{\emph{vertex-set}} $V$ is a finite non-empty set and the \textbf{{\em arc-set}} $A\subseteq V\times V$ is a binary relation on $V$. The elements of $V$ and $A$ are called \textbf{\emph{vertices}} and \textbf{\emph{arcs}} of $\vGa$, respectively. An \textbf{\emph{automorphism}} of $\vGa$ is a permutation $\sigma$ of $V$ that preserves the relation $A$, that is, $(x^\sigma,y^\sigma)\in A$ for every $(x,y)\in A$. When the binary relation $A$ is symmetric, we say that $\Gamma$ is a \textbf{\emph{graph}}.

Let $R$ be a group and let $S$ be a subset of $R$. The \textbf{\emph{Cayley digraph}} on $R$ with connection set $S$, denoted $\mathrm{Cay}(R,S)$, is the digraph with vertex-set $R$ and with $(g,h)$ being an arc if and only if $hg^{-1}\in S$.  The group $R$ acts regularly (via its right regular representation) as a group of automorphisms of $\mathrm{Cay}(R,S)$  and hence $R\le \Aut(\mathrm{Cay}(R,S))$, where we identify $R$ with its image under its right regular representation. The Cayley digraph $\mathrm{Cay}(R,S)$ is a graph if and only if $S$ is inverse closed, that is, $S=S^{-1}$ where $S^{-1}:=\{s^{-1}\mid s\in S\}$.

When the automorphism group of  $\mathrm{Cay}(R,S)$ equals $R$, the Cayley digraph $\mathrm{Cay}(R,S)$ is called a \textbf{\emph{DRR}}, for digraphical regular representation. Additionally, if $S$ is inverse closed, $\mathrm{Cay}(R,S)$ is called a \textbf{\emph{GRR}}, for graphical regular representation. 

Babai \cite{Babai} has proved that except for five small groups, every finite group admits a DRR. The five exceptions are  the quaternion group $Q_8$ of order $8$ and the abelian groups $C_2^2$, $C_2^3$, $C_2^4$ and $C_3^2$. Godsil completed the characterization of groups admitting a GRR~\cite{Godsil} very shortly before Babai's work on DRRs; this characterization includes infinitely many exceptions.

After the classification of finite groups admitting a DRR or a GRR was completed, researchers proposed and investigated various natural generalizations. These investigations endeavoured to classify graphical representations of groups in restricted classes of digraphs or graphs. For instance, Babai and Imrich~\cite{BabaiI} classified finite groups admitting a tournament regular representation. Morris and Spiga~\cite{MorrisSpiga1,MorrisSpiga3,Spiga}, answering a question of Babai~\cite{Babai}, classified the finite groups admitting an oriented regular representation.
For more results generalising DRRs in various directions, we refer to
\cite{DFS,DFS1,DFR,GFR,Xiaf}.

The scope of this paper is to obtain a classification of the finite groups admitting a graphical representation as Haar graphs. This classification allows us to solve a problem proposed  by Feng, Kov\'{a}cs, Wang, and Yang~\cite{FKWY} and to provide an alternative solution to a problem asked by Est\'elyi and Pisanski~\cite{EP} that was first solved in~\cite{FKWY}. In addition, we are able to use our classification to improve a 1980 result of Babai on the representation of finite groups on posets and achieve an improvement to his related result on the representation of finite groups on distributive lattices~\cite{Babai}.

\subsection{Haar graphs}\label{sec:haar}
Let $R$ be a finite group and let $S$ be a subset of $R$. The \textbf{\textit{Haar graph}} on $R$ with connection set $S$ has vertex set
$$R\times \{-1,1\} = (R\times \{-1\})\cup (R\times \{1\}),$$
where $\{(g,-1),(h,1)\}$ is declared to be an edge if and only if $hg^{-1} \in S$. We
denote by $\mathrm{Haar}(R,S)$ the Haar graph on $R$ with connection set $S$. Observe that $\mathrm{Haar}(R,S)$
is bipartite with parts $R \times \{-1\}$, $R \times \{1\}$. Moreover, for each $g\in R$, the mapping 
\begin{alignat*}{2}
\rho_g:R\times\{-1,1\}&\longrightarrow R\times\{-1,1\}\\
(x,\varepsilon)&\longmapsto (xg,\varepsilon)
\end{alignat*}
is a permutation of $R\times\{-1,1\}$ and is an automorphism of $\mathrm{Haar}(R,S)$. In particular,
$$\hat{R}:=\{\rho_g\mid g\in R\}$$
is a subgroup of $\mathrm{Aut}(\mathrm{Haar}(R,S))$ isomorphic to $R$ and having orbits exactly $R \times \{-1\}$ and $R \times \{1\}$. The action of $\hat R$ on both of these orbits  is regular.
We say that $\mathrm{Haar}(R,S)$ is a \textbf{\textit{Haar Graphical Representation}}, or HGR for short, if
$$\hat{R} = \mathrm{Aut}(\mathrm{Haar}(R,S)).$$

Abelian groups do not admit HGRs, see for instance~\cite{EP}. Indeed, when $R$ is an abelian group, for every subset $S$ of $R$, the mapping
\begin{equation}\label{iota}
\iota:(x,\varepsilon)\mapsto(x^\iota,-\varepsilon),\quad \forall (x,\varepsilon)\in R\times\{-1,1\},
\end{equation}
is a graph automorphism of $\mathrm{Haar}(R,S)$, turning $\mathrm{Haar}(R,S)$ into a Cayley graph. 
Several papers have studied conditions that serve to confirm or deny that all Haar graphs on a particular group are Cayley graphs, for instance~\cite{EP, FKWY, FKY}. In~\cite[Problem 1.6]{FKWY}, Feng, Kov\'{a}cs, Wang, and Yang ask for a characterisation of finite nonabelian groups that admit HGRs.

One of the main results of this paper is to provide such a classification. 
\begin{theorem}\label{thrm:mainHRR}
Let $R$ be a finite group. Then one of the following holds
\begin{enumerate}
\item\label{thrm:mainHRR1}$R$ admits a Haar graphical representation,
\item\label{thrm:mainHRR2}$R$ is an abelian group and there exists $S\subseteq R$ with $\hat R\rtimes\langle\iota\rangle=\mathrm{Aut}(\mathrm{Haar}(R,S))$, where $\iota$ is defined in~$\eqref{iota}$, 
\item\label{thrm:mainHRR3}$R$ is one of the twenty-two exceptional groups appearing in Table~$\ref{auxiliary}$.
\end{enumerate}
\end{theorem}
\begin{table}[!ht]
\begin{tabular}{c|c}\hline
Order&Group\\\hline
3&$C_3$ \\
4&$C_2^2$ or $C_4$ \\
5&$C_5$\\
6&$C_6$ or $D_6$\\
7&$C_7$\\
8&$C_2^3$, or $C_4\times C_2$, or $Q_8$, or $D_8$\\
9&$C_{3}^2$ \\
10&$D_{10}$\\
12&$\mathrm{Alt}(4)$ or $D_{12}$ or $\langle x,y\mid x^6=y^4=1, x^3=y^2, y^{-1}xy=x^{-1} \rangle$\\
14&$D_{14}$ \\
16&$C_{2}^4$ or $C_4\times C_2^2$ or  $Q_8\times C_2$\\
18&$\langle e_1,e_2,x\mid e_1^3=e_2^3=x^2=[e_1,e_2]=1, e_1^x=e_1^{-1},e_2^x=e_2^{-1}\rangle$\\
32&$C_{2}^5$\\
\hline
\end{tabular}
\caption{Exceptional examples for Theorem~\ref{thrm:mainHRR}. Here $D_n$ denotes the dihedral group of order $n$.}\label{auxiliary}
\end{table}

A classification of groups admitting $m$-partite digraphical representations was obtained in~\cite{DJS}. This generalised a previous result on Haar digraphical representations~\cite{DJSalso}. Similar to the situation with DRRs and GRRs, constructions for Haar digraphical representations were easier to find than our related results for graphs, and there are fewer exceptional groups. This is because there is more flexibility available in constructing digraphs that admit specified automorphisms than in constructing graphs admitting those automorphisms.

Theorem~\ref{thrm:mainHRR} allows us  to provide an alternative solution to a question of Est\'elyi and Pisanski~\cite[Problem~1]{EP}, first answered in~\cite{FKWY}. It may happen that, given a group $R$ and a subset $S$, the Haar graph $\mathrm{Haar}(R,S)$ admits automorphisms that exchange the parts of the bipartition. Together with the action of $\hat R$ these may form a regular subgroup of $\Aut(\Haar(R,S))$, making $\Haar(R,S)$ a Cayley graph. As we observed above, this always happens when $R$ is an abelian group. Est\'elyi and Pisanski  proposed the problem: 
\begin{quote}Determine the finite  groups $R$ for which all Haar graphs $\mathrm{Haar}(R,S)$ are Cayley graphs.
\end{quote} A partial answer to this question is in~\cite{FKY}, and a complete answer in~\cite{FKWY}; however, the constructions provided in those papers are not HGRs. The solution to this problem is an easy corollary of Theorem~\ref{thrm:mainHRR}.
\begin{corollary}\label{corollaryEP}
Let $R$ be a finite group. If every Haar graph on $R$ is a Cayley graph, then one of the following holds
\begin{enumerate}
\item $R$ is abelian,
\item\label{part2} $R$ is isomorphic to $D_6$, $D_8$, $Q_8$, $D_{10}$, or $Q_8\times C_2$.
\end{enumerate}
\end{corollary}
\begin{proof}
Let $R$ be a non-abelian group such that every Haar graph on $R$ is a Cayley graph.
From Theorem~\ref{thrm:mainHRR}, $R$ is one of the ten nonabelian groups in Table~\ref{auxiliary}. We have used the computer algebra system magma~\cite{magma} to verify that the only nonabelian groups in Table~\ref{auxiliary} such that each Haar graph is a Cayley graph are the groups listed in part~\eqref{part2}.
\end{proof}

We have an additional application of our investigation, which strengthens some results of Babai~\cite{Babai}. For this application we need a corollary of Theorem~\ref{thrm:mainHRR}. 
\begin{corollary}\label{thrm:mainHRRplus}
Let $R$ be a finite group. Then one of the following holds:
\begin{enumerate}
\item\label{thrm:mainHRR11}there exists a subset $S$ of $R$ such that $\hat R$ is the stabiliser in $\mathrm{Aut}(\mathrm{Haar}(R,S))$ of the bipartition $\{R\times\{-1\},R\times\{1\}\}$,
\item\label{thrm:mainHRR33}$R$ is one of the sixteen exceptional groups appearing in Table~$\ref{auxiliaryy}$.
\end{enumerate}
\end{corollary}
\begin{proof}
If $R$ satisfies part~\eqref{thrm:mainHRR1} or~\eqref{thrm:mainHRR2} of Theorem~\ref{thrm:mainHRR}, then part~\eqref{thrm:mainHRR11} follows.
If $R$ satisfies part~\eqref{thrm:mainHRR3}, then the proof follows with a computer computation.
\end{proof}

Corollary~\ref{thrm:mainHRRplus} allows to disregard abelian groups as being exceptional. Indeed, we no longer insist that $\mathrm{Aut}(\mathrm{Haar}(R,S))$ equals $\hat{R}$, but we do ask that the only automorphisms of $\mathrm{Haar}(R,S)$ fixing the natural bipartition $\{R\times\{-1\},R\times\{1\}\}$ are the elements of $\hat R$.
\begin{table}[!ht]
\begin{tabular}{c|c}\hline
Order&Group\\\hline
3&$C_3$ \\
4&$C_2^2$ or $C_4$ \\
5&$C_5$\\
6&$C_6$ or $D_6$\\
7&$C_7$\\
8&$C_2^3$, or $C_4\times C_2$, or $Q_8$, or $D_8$\\
9&$C_{3}^2$ \\
10&$D_{10}$\\
16&$C_{2}^4$ or $C_4\times C_2^2$ \\
32&$C_{2}^5$ \\\hline
\end{tabular}
\caption{Exceptional examples for Theorem~\ref{thrm:mainHRRplus}}\label{auxiliaryy}
\end{table}

\subsection{Representations of groups on other combinatorial structures}\label{abc:poset}
Among the early work on regular representations, Babai considered many combinatorial structures in addition to graphs. For some of these objects, it would not be reasonable to look for a regular representation, so a representation with few orbits is the closest parallel, especially if the representation can be made  semiregular. 

A \textit{\textbf{poset representation}} of a group $R$  is a partially ordered set $P$ with $R\cong \mathrm{Aut}(P )$.  The poset representation is called \textit{\textbf{semiregular}} if the action of $R$ on the underlying set
of $P$ is semiregular, that is, the stabiliser in $R$ of each point $x \in P$ is the identity. As any automorphism should preserve the partial order, it can never act regularly if the poset includes more than two elements.

Birkhoff~\cite{birkhoff}  proved in 1946 that every finite group $R$ has a semiregular poset representation with $|R| + 1$ orbits. Babai~\cite[Corollary~4.3]{Babai} significantly improved this by showing that, except for $Q_8$, $C_2^2$, $C_2^3$, $C_2^4$ and $C_3^2$\footnote{These are the five finite groups not admitting a DRR.},  every finite group admits a semiregular poset representation with three orbits. For some recent progress in this area, also involving infinite groups, we refer to the work of Barmak~\cite{Barmak}.

 Applying our results on Haar graphical representations we give the best possible result in this direction.

\begin{corollary}\label{corollary}
Let $R$ be a finite group. Then one (and only one) of the following holds:
\begin{enumerate}
\item\label{P1}\label{corollary1}$|R|\le 2$ and $R$ admits a semiregular poset representation with $1$ orbit;
\item\label{P2}\label{corollary2}$R$ is one of the fourteen groups in Table~$\ref{auxiliaryyy}$  and $R$ admits a semiregular poset representation with $3$ orbits;
\item\label{P3}\label{corollary4}$R$ is isomorphic to $C_2^2$ or to $C_2^3$, and $R$ admits a semiregular poset representation with $4$ orbits; or
\item\label{P4}in all other cases $R$ admits a semiregular poset representation with $2$ orbits.
\end{enumerate}
\end{corollary}
\begin{proof}
Suppose that $R$ admits a Haar graph $\mathrm{Haar}(R,S)$ such that $\hat R$ is the stabilizer in $\mathrm{Aut}(\mathrm{Haar}(R,S))$ of the bipartition $\{R\times \{-1\},R\times \{1\}\}$. We can turn $\mathrm{Haar}(R,S)$ into a poset $P$ having underlying set $R\times\{-1,1\}$ by defining $(a,s)\le (b,t)$ if and only if $(a,s)=(b,t)$, or $s=-1$, $t=1$ and $\{(a,-1),(b,1)\}$ is an edge of $\mathrm{Haar}(R,S)$. Clearly, $\mathrm{Aut}(P)$ is the subgroup of $\mathrm{Aut}(\mathrm{Haar}(R,S))$ fixing the bipartition $\{R\times \{-1\},R\times \{1\}\}$. In particular, except when $|R|\le 2$, each of these groups satisfies part~\eqref{P4}. Moreover, when $|R|\le 2$, $R$ admits a poset semiregular representation with $1$ orbit. Therefore, in view of Corollary~\ref{thrm:mainHRRplus}, we may suppose that $R$ is one of the sixteen groups in Table~\ref{auxiliaryy}. We have investigated the remaining groups with the help of a computer.
\end{proof}
\begin{table}[!ht]
\begin{tabular}{c|c}\hline
Order&Group\\\hline
3&$C_3$ \\
4&$C_4$ \\
5&$C_5$\\
6&$C_6$ or $D_6$\\
7&$C_7$\\
8&$C_4\times C_2$, or $Q_8$, or $D_8$\\
9&$C_{3}^2$ \\
10&$D_{10}$\\
16&$C_{2}^4$ or $C_4\times C_2^2$ \\
32&$C_{2}^5$ \\\hline
\end{tabular}
\caption{Exceptional examples for Corollary~\ref{corollary}}\label{auxiliaryyy}
\end{table}

Following the proof of~\cite[Corollary~4.5]{Babai} (in which a bound of $8^{|R|}$ is achieved), we obtain the following related result on
distributive lattices.

\begin{corollary}\label{corollaryy}
Given a finite group $R$, there exists a distributive lattice $L$ with $\mathrm{Aut}(L)\cong R$ and with $|L|\le 2^{\frac{3|R|}{2}}$.
\end{corollary}
\begin{proof}
We use Theorem~\ref{thrm:mainHRRRR}.
Suppose that $R$ satisfies either part~\eqref{thrm:mainHRRR11} or~\eqref{thrm:mainHRRR22} of Theorem~\ref{thrm:mainHRRRR}. In particular, there exists a subset $S$ of $R$ such that  $4\le |S|\le (|R|-6)/2$ and $\hat R$ is the stabilizer in $\mathrm{Aut}(\mathrm{Haar}(R,S))$ of the bipartition $\{R\times\{-1\},R\times\{1\}\}$. Let $T=R\setminus S$ and consider $\mathrm{Haar}(R,T)$. By Lemma~\ref{lem:bipartite complement}, $\hat R$ is the stabilizer in $\mathrm{Aut}(\mathrm{Haar}(R,T))$ of the bipartition $\{R\times\{-1\},R\times\{1\}\}$.

As in the proof of Corollary~\ref{corollary}, we turn $\mathrm{Haar}(R,T)$ into a poset $P$. The points in $P$ are the elements of $R\times\{-1,1\}$ where $(a,s)\le (b,t)$ if and only if $(a,s)=(b,t)$, or $s=-1$, $t=1$ and $\{(a,-1),(b,1)\}$ is an edge of $\mathrm{Haar}(R,T)$. Clearly, $\mathrm{Aut}(P)=\hat R$.

 Let $L$ be the lattice of ideals of $P$.\footnote{Recall that $I$ is an ideal of $P$ if, for each $a\in P$ and for each $b\in I$, $a\le b$ implies $a\in I$.} Clearly $L$ is distributive and by~\cite{birkhoff} we have $\mathrm{Aut}(L)\cong\mathrm{Aut}(P)\cong R$. Moreover,  we get 
 $$|L|\le 2+2^{|R|}-1+(2^{|R|}-1)2^{\frac{|R|}{2}-3}\le 2^{\frac{3|R|}{2}}.\footnote{Here, the summand $2$ accounts for the ideal generated by the empty set and by the whole of $L$. The summand $2^{|R|}-1$ accounts for an ideal generated by a subset of the layer $R\times\{-1\}$. Whereas, the summand $(2^{|R|}-1)2^{|R|/2-3}$ accounts for the remaining ideals and takes into account that each element in the layer $R\times\{1\}$ has at least $|R|/2+3$ elements below in the layer $R\times\{-1\}$.}$$ In particular, the proof follows  except when $R$ is a group in part~\eqref{thrm:mainHRRR33} or~\eqref{thrm:mainHRRR44} of Theorem~\ref{thrm:mainHRRRR}.

The veracity of the statement for the groups in part~\eqref{thrm:mainHRRR33} or~\eqref{thrm:mainHRRR44} of Theorem~\ref{thrm:mainHRRRR}  has been verified with the help of a computer: for each group we have verified the existence of a poset $P$ with $\mathrm{Aut}(P)\cong R$ and such that the lattice of ideals of $P$ has cardinality at most $2^{3|R|/2}$.
\end{proof}
Except for the base of the exponent, Corollary~\ref{corollaryy} is best possible. Indeed, by~\cite[Proposition~4.8]{Babai}, if $R$ is a cyclic group of prime order $p$ and $L$ is a distributive lattice admitting $R$ as group of automorphisms, then $|L|\ge 2^p$.

\subsection{Abundance of Haar graphical representations}
Alongside classifying groups that admit a regular representation, it is natural to ask how common these representations are among the class of combinatorial objects being studied.\footnote{Typically for this question, we consider the frequency of such objects among objects in the class whose automorphism group contains a fixed regular subgroup.} Accordingly, at around the same time that the classification of DRRs and GRRs was completed, Babai and Godsil~\cite[Conjecture~3.13]{BaGo} conjectured that
the proportion of subsets $S$ of $R$ such that $\mathrm{Cay}(R,S)$ is a $\mathrm{DRR}$ goes to $1$ as $|R|\to\infty$, that is,
$$\lim_{|R|\to\infty}\frac{|\{S\subseteq R\mid \mathrm{Cay}(R,S) \hbox{ is a DRR}\}|}{2^{|R|}}=1.$$ This conjecture was proved only recently in~\cite{MorrisSpiga} for digraphs. The methods from~\cite{MorrisSpiga} were then improved by Xia and Zheng~\cite{XiaZheng} to prove the analogous conjecture for graphs, see~\cite{XiaZheng} for more details. Using the ideas in~\cite{MorrisSpiga}, together with Gan and Xia, the second author has just completed an asymptotic enumeration of Haar graphical representations \cite{Gan}. Indeed,
$$\lim_{\substack{R \hbox{ non-abelian}\\|R|\to\infty}}\frac{|\{S\subseteq R\mid \mathrm{Haar}(R,S) \hbox{ is a HGR}\}|}{2^{|R|}}=1.$$
Although valuable and interesting in their own context, the results in~\cite{Gan} are non-constructive and not sufficient to give much insight into the existence of HGRs; in fact, they only prove the existence of HGRs over non-abelian groups, when $|R|\ge 2^{83}$.

\section{Elementary results}

In this section, we present some relatively elementary facts about the structure of Haar graphs. These will be important in our main results. 

The statement of Theorem~\ref{thrm:mainHRR} and its corollaries hinted at the importance of understanding the subgroup that stabilises the natural bipartition within the automorphism group of a Haar graph. Indeed, as this is an index-2 subgroup of the automorphism group, its importance to the overall group structure is not surprising. Since this subgroup will arise often in our proofs, we introduce notation for it.

 Given a Haar graph $\mathrm{Haar}(R,S)$, we let $\mathrm{Aut}(\mathrm{Haar}(R,S))_0$ be the subgroup of the automorphism group $\mathrm{Aut}(\mathrm{Haar}(R,S))$ fixing the natural bipartition $\{R\times\{-1\},R\times\{1\}\}$ of the vertex set of $\mathrm{Haar}(R,S)$. Clearly, $\hat R\le \mathrm{Aut}(\mathrm{Haar}(R,S))_0$.

In order to have some control over the cardinality of $|S|$ in $\Haar(R,S)$, we will sometimes want the option of working with the bipartite complement of $\Haar(R,S)$ rather than with $\Haar(R,S)$ itself. The following lemma shows that the bipartite complement is also a Haar graph, and for our purposes will always have the same automorphism group.\footnote{We actually show only that the subgroups consisting of automorphisms that fix the natural bipartition are equal, but in connected Haar graphs this is equivalent to the automorphism groups being equal.}

\begin{lemma}\label{lem:bipartite complement}
The bipartite complement of the Haar graph $\Haar(R,S)$ with respect to the bipartition $\{R\times\{-1\},R\times\{1\}\}$  is the Haar graph 
$\mathrm{Haar}(R,R\setminus S)$. 

Furthermore, $\Aut(\mathrm{Haar}(R,S))_0=\Aut(\mathrm{Haar}(R,R\setminus S))_0$.
\end{lemma}

\begin{proof}
The first part of the lemma follows directly from the definition of Haar graph.

The ``furthermore" is equally clear. We can partition the pairs of vertices of $\Haar(R,S)$ into three sets: those pairs $\{(r_1,i),(r_2,i)\}$ where $i \in \{-1,1\}$ and $r_1, r_2 \in R$; those pairs $\{(r_1,-1),(r_2,1)\}$ for which $r_1, r_2 \in R$ and $r_2r_1^{-1} \in S$ (these are joined by an edge in $\Haar(R,S)$); and those pairs $\{(r_1,-1),(r_2,1)\}$ for which $r_1, r_2 \in R$ and $r_2r_1^{-1} \notin S$ (these are not joined by an edge in $\Haar(R,S)$).

If an automorphism fixes $R\times \{-1\}$ and $R\times \{1\}$ then it certainly fixes (setwise) the first of these three sets of pairs. If it is an 
 element of $\Aut(\mathrm{Haar}(R,S))$ then it must also preserve (setwise) the second set of pairs, and the third set of pairs. But the third set of pairs are exactly the edges of $\Haar(R, R\setminus S)$.
\end{proof}

When $\mathrm{Haar}(R,S)$ is connected, every automorphism preserves the natural bipartition\footnote{Preserving the bipartition includes the possibility that it exchanges $R\times \{-1\}$ with $R\times \{1\}$.}. We will always want to construct connected Haar graphs, so that the natural bipartition is preserved. This is an enormously important restriction in order to limit the permutations that could be graph automorphisms. It is therefore essential to be able to determine whether or not $\Haar(R,S)$ is connected directly and easily when we know $S$.

\begin{lemma}\label{connection2}
Let $S$ be a subset of the finite group $R$. Then $\mathrm{Haar}(R,S)$ is connected if and only if $S^{-1}S=\{s^{-1}s'\mid s,s'\in S\}$ generates $R$. 
\end{lemma}
\begin{proof}
The neighbours of $(1,-1)$ in $\mathrm{Haar}(R,S)$ are the elements of $S\times\{1\}$, and the vertices at distance two from $(1,-1)$ are the elements in $S^{-1}S\times\{-1\}$. 
In general, arguing by induction, we see that the elements at distance $2\ell$ from $(1,-1)$ are the vertices $(x,-1)$ where
$$x\in (S^{-1}S)^\ell=\{s_1^{-1}s_1'\cdots s_\ell'^{-1}s_\ell'\mid s_1,s_1',\ldots,s_\ell,s_\ell'\in S\}.$$
This shows that $\mathrm{Haar}(R,S)$ is connected if and only if $(S^{-1}S)^\ell=R$, for some $\ell\in \mathbb{N}$. In turn, this happens if and only if $S^{-1}S$ generates $R$.
\end{proof}

Sometimes we may not be given $S$ explicitly. Our next result tells us that in general, if $\Haar(R,S)$ is disconnected then there are  automorphisms of $\Haar(R,S)$ that fix the natural bipartition and are not elements of $\hat R$. We will primarily use this inductively. If $N \le R$ then $\Haar(N, S\cap N)$ appears as an induced subgraph of $\Haar(R,S)$ on the vertices $Nr\times \{-1,1\}$ for any $r \in R$. If we are able to assume inductively that $\Aut(\Haar(N, S\cap N))_0=\hat N$ (and $|N| \ge 3$), then Lemma~\ref{lem:disconnected} tells us that all of these subgraphs are connected, making it possible to deduce that $\Haar(R,S)$ is connected without explicit information about $S\cap N$.

\begin{lemma}\label{lem:disconnected}
If $|R|\ge 3$ and $\Haar(R,S)$ or its complement is disconnected, then $$\Aut(\Haar(R,S))_0>\hat R.$$ 
\end{lemma}

\begin{proof}
We assume that $\Haar(R,S)$ or its complement is disconnected and we let $C_1,\ldots,C_\ell$ be the connected components. For each $\varepsilon\in \{-1,1\}$, the sets $$C_1\cap (R\times\{\varepsilon\}),\ldots, C_\ell\cap (R\times\{\varepsilon\})$$ are blocks of imprimitivity for the transitive action of $\mathrm{Aut}(\mathrm{Haar}(R,S))_0$ on $R\times\{\varepsilon\}$. As $\hat R\le \Aut(\Haar(R,S))_0$ and as $\hat R$ acts regularly on $R\times\{\varepsilon\}$, these blocks are right cosets of some subgroup $R_1$ of $R$. We consider two possibilities: $|R_1|=1$, or $|R_1|>1$.

If $|R_1|=1$, then each connected component $C_\kappa$ is either a single vertex, or a pair of adjacent vertices, one from each set of the bipartition $\{R\times\{-1\},R\times\{1\}\}$. Thus $\Haar(R,S)$ or its complement is either an empty graph, or a perfect matching. As $|R|\ge 3$, from this the proof immediately follows.

If $|R_1|>1$, then we can let $\hat R_1$ act semiregularly on the vertices of one connected component, while fixing every other vertex. This subgroup is in $\Aut(\Haar(R,S))_0$,  giving a non-identity element of a point stabiliser of $\Aut(\Haar(R,S))_0$.
\end{proof}

%
%
%
%
%

Lemma~\ref{lem:done if nbrs fixed} will help to avoid some repetitions in our inductive arguments. 

\begin{lemma}\label{lem:done if nbrs fixed}
If $\mathrm{Haar}(R,S)$ is connected, $\Aut(\Haar(R,S))_{(1,-1)}$ fixes  $S \times \{1\}$ pointwise, and $\Aut(\Haar(R,S))_{(1,1)}$ fixes  $S^{-1} \times \{-1\}$ pointwise, then $\Aut(\Haar(R,S))_0=\hat R$.
\end{lemma}

\begin{proof}
Observe that the hypothesis implies that, for every vertex $(r,\varepsilon)\in R\times\{-1,1\}$, $\mathrm{Aut}(\mathrm{Haar}(R,S))_{(r,\varepsilon)}$ fixes  the neighbourhood of $(r,\varepsilon)$ pointwise.

Let $g \in \Aut(\Haar(R,S))_0$. Replacing $g$ with $g\rho_{r^{-1}}$ for some $\rho_{r^{-1}}\in \hat R$, we may suppose that $g$ fixes $(1,-1)$. As $g$ fixes $(1,-1)$, $g$ fixes the neighbourhood of $(1,-1)$ pointwise. Now an inductive argument, pivoting on the connectedness of $\mathrm{Haar}(R,S)$, implies that $g$ is the identity permutation.
\end{proof}

The following result will also be important in our inductive arguments. This shows that from a graph automorphism of $\Haar(R,S)$ that exchanges the parts of the bipartition, we can find a closely-related group automorphism of $R$. This has implications for the relations among elements of $R$.

\begin{lemma}\label{group automorphism}
Suppose $\mathrm{Haar}(R,S)$ is connected, $\mathrm{Aut}(\mathrm{Haar}(R,S))_0=\hat R$, and there exists $\varphi\in\mathrm{Aut}(\mathrm{Haar}(R,S))$ with $(1,-1)^\varphi=(1,1)$. Then there exists $\hat \varphi\in\mathrm{Aut}(R)$ and $x\in R$ with
$(r,-1)^\varphi=(r^{{\hat \varphi}},1)$ and $(r,1)^\varphi=(r^{{\hat \varphi}^{-1}}x,-1)$, $\forall r\in R$.
\end{lemma}
\begin{proof}
Since $\mathrm{Haar}(R,S)$  is connected and since $(1,-1)^\varphi=(1,1)$, we deduce that $\varphi$ maps $R\times\{-1\}$ to $R\times\{1\}$. We let $\hat \varphi:R\to R$ be the mapping such that $(r,-1)^\varphi=(r^{\hat \varphi},1)$, $\forall r\in R$.

Let $r\in R$ and consider $g=\varphi^{-1}(r,r)\varphi((r^{\hat \varphi})^{-1},(r^{\hat \varphi})^{-1})\in \mathrm{Aut}(\mathrm{Haar}(R,S))$. As $g$ fixes the bipartition $\{R\times\{-1\},R\times\{1\}\}$, we get $g\in \mathrm{Aut}(\mathrm{Haar}(R,S))_0=\hat R$. Moreover, as 
\begin{align*}
(1,1)^g&=
(1,1)^{\varphi^{-1}(r,r)\varphi((r^{\hat \varphi})^{-1},(r^{\hat \varphi})^{-1})}=(1,-1)^{(r,r)\varphi((r^{\hat \varphi})^{-1},(r^{\hat \varphi})^{-1})}\\
&=(r,-1)^{\varphi((r^{\hat \varphi})^{-1},(r^{\hat \varphi})^{-1})}=(r^{\hat \varphi},1)^{((r^{\hat \varphi})^{-1},(r^{\hat \varphi})^{-1})}=(1,1),
\end{align*}
we deduce that $g$ is the identity permutation. Therefore, $\varphi^{-1}(r,r)\varphi=(r^{\hat \varphi},r^{\hat \varphi})$ and $\hat\varphi\in\mathrm{Aut}(R)$. This shows $(r,-1)^\varphi=(r^{\hat \varphi},-1)$, $\forall r\in R$.

Let $x\in R$ with $(1,1)^\varphi=(x,-1)$. As $\varphi^2\in\mathrm{Aut}(\mathrm{Haar}(R,S))_0=\hat R$, we deduce $\varphi^2=(x,x)\in \hat R$. We now have 
\begin{align*}
(r,1)^\varphi&=(1,1)^{(r,r)\varphi}=(1,-1)^{\varphi (r,r)\varphi}=(1,-1)^{\varphi(r,r)\varphi^{-1}\cdot \varphi^2}\\
&=(1,-1)^{(r^{{\hat \varphi}^{-1}},r^{{\hat \varphi}^{-1}})(x,x)}=(r^{{\hat\varphi}^{-1}},-1)^{(x,x)}=(r^{{\hat\varphi}^{-1}}x,-1).
\end{align*}
\end{proof}

Although the final result we present here is purely group-theoretic, as it is essentially brief background material that will be  necessary to our main arguments, we include it in this section. 
\begin{lemma}\label{lemma:Rsimple}
Let $R$ be a non-abelian simple group. Then one of the following holds
\begin{enumerate}
\item\label{eq:Rsimple1} there exist $s,t\in R$  with $R=\langle s,t\rangle$ and ${\bf o}(s)\ge 12$, 
\item\label{eq:Rsimple2} $R$ is isomorphic to $\mathrm{Alt}(5)$, $\mathrm{Alt}(6)$, $\mathrm{Alt}(7)$,  $\mathrm{PSL}_2(7)$, $\mathrm{PSL}_2(8)$, $\mathrm{PSL}_2(11)$, $\mathrm{PSL}_3(4)$, $\mathrm{PSU}_3(5)$, $M_{11}$, $M_{12}$, or $M_{22}$.
\end{enumerate}
In all cases, there exist $s,t\in R$  with $R=\langle s,t\rangle$ and ${\bf o}(s)\ge 5$.
\end{lemma}
\begin{proof}
In~\cite{GuralnickKantor}, it is proved that every non-identity element of a finite non-abelian simple group belongs to a pair of generating elements. From this, we see that part~\eqref{eq:Rsimple1} is satisfied as long as $R$ contains an element having order at least $12$. Therefore, suppose that the maximal order of an element of $R$ is less than $12$.

We use the Classification of the Finite Simple Groups. If $R$ is an alternating group, then part~\eqref{eq:Rsimple2} immediately follows because $(1\,2\,3)(4\,5\,6\,7\,8)$ has order $15$ and lies in $\mathrm{Alt}(8)$. If $R$ is a sporadic simple group, then by~\cite{atlas} $R$ contains an element of order at least $12$, except when $R$ is $M_{11}$, $M_{12}$ or $M_{22}$. When $R$ is a simple group of Lie type, the maximal order of an element is tabulated in~\cite{KantorSeress} and we see that also in this case part~\eqref{eq:Rsimple2} is satisfied.\footnote{The work of Kantor and Seress gives a complete description of the maximal element order of all simple groups of Lie type, except for $\mathrm{Sp}_n(q)$, $\Omega_n^+(q)$ and $\Omega_n^-(q)$ with $q$ even. Nevertheless, the bounds obtained in~\cite{KantorSeress} are sufficient for our purposes. Incidentally, the maximal element order in even characteristic symplectic groups was obtained in~\cite{spiga1}.}
\end{proof}
With this, we are ready to present our constructions.

\section{Explicit constructions for cyclic and $2$-generated groups}\label{constructions}

In this section we provide the base cases for the inductive argument that follows. More precisely, we provide a construction on many groups that are cyclic or 2-generated, for a Haar graph that is an HGR if $R$ is nonabelian, and whose automorphism group is $\hat R \rtimes \langle \iota\rangle$ if $R$ is abelian.

We begin with a construction for a Haar graph on any cyclic group that has a very small automorphism group.

Whenever $i \ge 1$,  $\mathbf{o}(s) \ge \min\{9,i+6\}$, and $R=\langle s \rangle$, it is possible to prove that $\Aut(\Haar(R, \{1,s,\ldots, s^i,s^{i+2}\})_0=\hat R$. However, for our proof we only need to establish this in one case, when $i=6$.

\begin{lemma}\label{cyclic2}
Let $s \in R$ with $\mathbf{o}(s) \ge 12$, and let $S=\{s^i: 0 \le i \le 6\} \cup \{s^8\}$. Then $\Aut(\Haar(R,S))=\hat R$.
\end{lemma}

\begin{proof}
From Lemma~\ref{connection2}, $\mathrm{Haar}(R,S)$ is connected and hence $\{R\times\{-1\},R\times\{1\}\}$ is the unique bipartition of $\mathrm{Haar}(R,S)$.

Observe that since $\mathbf{o}(s)\ge 12$, there are only $4$ vertices that have $6$ common neighbours with $(1,-1)$: namely $(s^{\pm1},-1)$ and $(s^{\pm2},-1)$ (any other vertex in $R\times \{-1\}$ has at least $3$ neighbours that are not in $S$). Furthermore, among the neighbours of $(1,-1)$, only $(s^8,1)$ has a unique neighbour in $\{(s^{\pm1},-1),(s^{\pm2},-1)\}$, and that neighbour is $(s^2,-1)$. Therefore, if $g \in \Aut(\Haar(R,S))_{(1,-1)}$ then $g$ fixes $(s^8,1)$ and $(s^2,-1)$. Also, only one of the four vertices that has $6$ common neighbours with $(s^2,-1)$ lies in $\{(s^{\pm1},-1),(s^{-2},-1)\}$: namely, $(s,-1)$. Therefore $(s,-1)$ must also be fixed by $g$.

Since the hypothesis that $g$ fixes $(1,-1)$ was sufficient to deduce that $g$ fixes $(s,-1)$, we can see inductively that $g$ fixes every vertex of $\langle s \rangle \times \{-1\}=R\times\{-1\}$, and since $g$ fixes $(s^8,1)\in R\times\{1\}$ we can likewise deduce that $g$ fixes every vertex of $\langle s \rangle \times \{1\}$.
\end{proof}

Lemma~\ref{cyclic2} allows us to establish the automorphism group of these Haar graphs over cyclic groups.

\begin{corollary}\label{cor:cyclic construction}
Let $R=\langle s\rangle$ be cyclic with $\mathbf{o}(s) \ge 12$ let $S=\{s^i: 0 \le i \le 6\} \cup \{s^8\}$. Then $\Aut(\Haar(R,S))=\hat R \rtimes \langle \iota\rangle,$ where $\iota$ is as defined in~\eqref{iota}.

Furthermore, if $\mathbf{o}(s) \ge 22$ then $4 \le |S|\le (|R|-6)/2$.
\end{corollary}

\begin{proof}
Let $A=\mathrm{Aut}(\mathrm{Haar}(R,S_i))$ where $i \in \{1,2\}$.
As $R$ is abelian, $\hat R\rtimes\langle\iota\rangle\le A$ and hence $\mathrm{Haar}(R,S)$ is vertex transitive.

By Lemma~\ref{cyclic2}, $A_{(1,-1)}=1$ and hence $A\le \hat R\rtimes \langle\iota\rangle$.

We have $|S_2|=8$, making the ``furthermore" an easy calculation.
\end{proof}

Lemma~\ref{cyclic2} also allows us to establish the existence of Haar graphs over non-cyclic $2$-generated groups having small automorphism group.

\begin{proposition}\label{prop:2gen}
Let $R=\langle s,t\rangle$ with ${\bf o}(s)\ge 12$ and with $s^t\neq s$, and let $S_1=\{s^i: 0 \le i \le 6\} \cup \{s^8\} \cup \{t,ts\}$ and $S_2=S_1 \cup\{ts^3\}$. When considering $S_1$ we also assume $s^t \neq s^{-1}$; for $S_2$ we assume $R$ is dihedral.
Then $$\mathrm{Aut}(\mathrm{Haar}(R,S_1))=\Aut(\Haar(R,S_2)=\hat R.$$

Furthermore, if $|R:\langle s\rangle|>2$ then $4 \le |S_1| \le (|R|-6)/2$, while if  $\mathbf{o}(s) \ge 14$ then $4 \le |S_2|\le (|R|-6)/2$.
\end{proposition}
\begin{proof}
Let $\Gamma_1=\Haar(R,S_1)$ and $\Gamma_2=\Haar(R,S_2)$. Let $A=\mathrm{Aut}(\Gamma_1)$ and let $B=\Aut(\Gamma_2)$. We begin by showing that in each of the two graphs,  every vertex of $R \times \{-1,1\}$ that has at least seven common neighbours with $(1,-1)$ lies in $\langle s \rangle \times \{-1\}$.

First note that any vertex that has common neighbours with $(1,-1)$ must lie in $R \times \{-1\}$. Suppose that in either graph, $(x,-1)$ has at least seven common neighbours with $(1,-1)$. Therefore, at least four of these common neighbours can be denoted $(s_i,1)$ for $1 \le i \le 4$, where each $s_i\in S\setminus\{t,ts\}\subseteq \langle s\rangle$. As each  $(s_i,1)$ is a neighbour of $(x,-1)$, we have $s_ix^{-1}\in S$. Since $s_i\in \langle s\rangle$ and  for each $j \in \{1,2\}$,  $S_j\setminus \langle s\rangle \subseteq \{t,ts,ts^3\}$, this is only possible if $x\in \langle s\rangle$.

In each graph, $(s^i,1)$ is a common neighbour of $(1,-1)$ and $(s,-1)$ for every $1 \le i \le 6$, and $(ts,1)$ is also a common neighbour, for a total of $7$ common neighbours.

We prove that the following sets are blocks of imprimitivity for $A$ and for $B$\footnote{This is a mild abuse of terminology since $A$ and/or $B$ may not be transitive.}:
\begin{itemize}
\item $R \times \{-1\}$, $R \times \{1\}$;
\item $\langle s\rangle g \times \{\varepsilon\}$, where $g \in R$ and $\varepsilon \in \{-1,1\}$; and
\item $\langle s \rangle g \times \{-1,1\}$, where $g \in R$.
\end{itemize}
From Lemma~\ref{connection2}, since $1 \in S_1,S_2$ and $\langle S_1 \rangle=\langle S_2\rangle=R$, $\Gamma_1$ and $\Gamma_2$ are connected. Therefore $R\times \{-1\}$ and $R \times \{1\}$ are uniquely determined by the parity of the length of any path from a fixed vertex. This shows that $R\times \{-1\}$ and $R \times \{1\}$ are blocks of imprimitivity for $A$ and for $B$.

By our argument above, in both $\Gamma_1$ and $\Gamma_2$, $(s,-1)$ has more common neighbours with $(1,-1)$ than any vertex outside $\langle s\rangle \times \{-1\}$. By taking the transitive closure of this relation, we obtain $\langle s \rangle \times \{-1\}$, which must therefore yield a block of imprimitivity for $A$ and for $B$. An analogous argument shows that $\langle s\rangle\times \{1\}$ is also a block of imprimitivity for  $A$ and for $B$.

For any given set $\langle s\rangle g \times \{\varepsilon\}$, in both $\Gamma_1$ and $\Gamma_2$ there is a unique other set with which each of the vertices of $\langle s\rangle g \times \{\varepsilon\}$ has more than three neighbours: $\langle s \rangle g \times \{-\varepsilon\}$. Thus $(\langle s\rangle\times\{-1\})\cup(\langle s\rangle\times\{1\})=\langle s\rangle\times\{-1,1\}$ is also a block of imprimitivity for $A$ and for $B$.

Let $\gamma\in A$ or $\gamma \in B$  with $(s^i,-1)^\gamma=(s^i,-1)$. Since $\langle s\rangle \times\{- 1,1\}$ and $\langle s\rangle \times\{ -1\}$  are blocks for $A$,
 $\gamma$ fixes each of  $\langle s \rangle \times \{-1\}$ and $\langle s \rangle  \times \{1\}$ setwise. So $\gamma$ induces an automorphism of the induced subgraph on $\langle s \rangle  \times \{ -1,1\}$, which is $\Haar(\langle s\rangle, S_j \cap \langle s\rangle)$ for some $j \in \{1,2\}$.
 Moreover, since $\gamma$ fixes some point of this set, it fixes the natural bipartition. Therefore, using Lemma~\ref{lem:bipartite complement}, $\gamma$ induces an automorphism of $\Haar(\langle s\rangle, \langle s\rangle \cap S_j)$ for each $j \in \{1,2\}$\footnote{In fact, $\langle s\rangle \cap S_1=\langle s\rangle \cap S_2$.}. Now Lemma~\ref{cyclic2}  implies that $\gamma$ fixes   $\langle s \rangle  \times \{-1,1 \}$ pointwise.
 At this point, we have established that $A_{(1,-1)}$ and $B_{(1,-1)}$ fix every point of $\langle s\rangle \times \{- 1,1\}$. An entirely analogous argument shows that $A_{(1,1)}$ and $B_{(1,1)}$ fix every point of $\langle s\rangle \times \{- 1,1\}$.
 
   Note that in $\Haar(R,S_i)$, $(t,1)$ is the only mutual neighbour of $(1,-1)$ and $(s^{-1},-1)$ outside $\langle s \rangle \times \{1\}$, while $(ts,1)$ is the only mutual neighbour of $(1,-1)$ and $(s,-1)$outside $\langle s \rangle \times \{1\}$. Therefore $A_{(1,-1)}$ and $B_{(1,-1)}$ must fix both $(t,1)$ and $(ts,1)$, and in the case of $B_{(1,-1)}$ also $(ts^3,1)$, since these are the only remaining neighbours of $(1,-1)$. This implies that  $A_{(1,-1)}$ fixes every point of $S_1 \times \{1\}$, and analogously, $A_{(1,1)}$ fixes every point of $S_1^{-1} \times \{-1\}$. Also, $B_{(1,-1)}$ fixes every point of $S_2 \times \{1\}$, and analogously, $B_{(1,1)}$ fixes every point of $S_2^{-1} \times \{-1\}$

We have now established all of the conditions of Lemma~\ref{lem:done if nbrs fixed}, so $A_0=B_0=\hat R$.

We now argue by contradiction and we suppose $A\ne \hat R$ or $B \neq \hat R$. Let $\varphi\in A\setminus \hat R$ or in $B \setminus \hat R$. Then $(1,-1)^\varphi\in R\times\{1\}$ and hence, replacing $\varphi$ if necessary, we may suppose that $(1,-1)^\varphi=(1,1)$.  By Lemma~\ref{group automorphism}, there exists $\hat\varphi\in\mathrm{Aut}(R)$ with $(r,-1)^\varphi=(r^{\hat \varphi},1)$, $\forall r\in R$. Since $\langle s\rangle\times\{-1,1\}$ is a block of imprimitivity for $A$ and for $B$ and since $\varphi$ maps the vertex $(1,-1)\in \langle s\rangle\times\{-1,1\}$ to a vertex in the same block, $\varphi$ fixes $\langle s\rangle\times\{-1,1\}$ setwise.
Therefore, $\varphi$ induces a graph automorphism on $\Haar(\langle s\rangle, S_j \cap \langle s\rangle)$ for each $j \in \{1,2\}$. Using Lemma~\ref{lem:bipartite complement} and Corollary~\ref{cor:cyclic construction}, we know that $\Aut(\Haar(\langle s\rangle, S_j\cap \langle s\rangle))=\widehat{\langle s \rangle} \rtimes \langle \iota\rangle$, where $(r,\varepsilon)^\iota=(r^{-1},-\varepsilon)$, $\forall (r,\varepsilon)\in \langle s\rangle\times\{-1,1\}$. As $(1,-1)^\varphi=(1,1)=(1,-1)^{\iota}$, the restriction of $\varphi$ to $\langle s\rangle\times\{-1,1\}$ equals $\iota$. In particular, $(1,1)^\varphi=(1,1)^{\iota}=(1,-1)$.

 Since $\varphi$ maps $(1,1)$ to $(1,-1)$, $\varphi$ maps the neighbourhood $S^{-1}\times\{-1\}$ of $(1,1)$ to the neighbourhood $S\times\{1\}$ of $(1,-1)$. However, since $\varphi$ fixes $\langle s\rangle\times\{-1,1\}$ setwise, we deduce in $\Gamma_1$ that
$$\{(t^{-1},-1),(s^{-1}t^{-1},-1)\}^\varphi=\{(t,1),(ts,1)\},$$
or in $\Gamma_2$ that 
$$\{(t^{-1},-1),(s^{-1}t^{-1},-1),(s^{-3}t^{-1},-1)\}^\varphi=\{(t,1),(ts,1),(ts^3,1)\}.$$
We consider the remaining cases individually. 

In $\Gamma_1$, assume first $(t^{-1},-1)^\varphi=(t,1)$ and $(s^{-1}t^{-1},-1)^\varphi=(ts,1)$. Thus $(t^{-1})^{\hat\varphi}=t$ and 
$$ts=(s^{-1}t^{-1})^{\hat \varphi}=(s^{-1})^{\hat\varphi}(t^{-1})^{\hat\varphi}=(s^{-1})^{-1} t=st.$$
This implies that $s^t=s$, which contradicts the fact that $s^t\notin \{s,s^{-1}\}$. Similarly, assume $(t^{-1},-1)^\varphi=(ts,1)$ and $(s^{-1}t^{-1},-1)^\varphi=(t,1)$. Thus $(t^{-1})^{\hat\varphi}=ts$ and 
$$t=(s^{-1}t^{-1})^{\hat \varphi}=(s^{-1})^{\hat\varphi}(t^{-1})^{\hat\varphi}=(s^{-1})^{-1} ts=sts.$$
This implies that $s^t=s^{-1}$, which  again contradicts the fact that $s^t\notin \{s,s^{-1}\}$.

Now in $\Gamma_2$, $R$ is dihedral, so $s^t=s^{-1}$.
Observe that $(t,1)$ has exactly three neighbours in $\langle s \rangle \times \{-1\}$: $(1,-1)$, $(s^{-1},-1)$, and $(s^{-3},-1)$. The preimages of these under $\varphi$ are $(1,1)$, $(s,1)$, and $(s^3,1)$ (using the known action of $\iota$). Therefore these must be neighbours of the preimage of $(t,1)$ under $\varphi$, which is one of $(t^{-1},-1)$, $(s^{-1}t^{-1},-1)$, or $(s^{-3}t^{-1},-1)$. 
However, none of $(t^{-1},-1)$, $(s^{-1}t^{-1},-1)$, or $(s^{-3}t^{-1},-1)$ is a neighbour of both $(s,1)$ and $(s^3,1)$.  Therefore no such $\varphi$ exists.

We certainly have $|S_1|, |S_2| \ge 4$. In fact, $|S_1|=10$. If $|R:\langle s \rangle |>2$ then since $\mathbf{o}(s) \ge 9$ we have $3\mathbf{o}(s)/2 \ge 13$, so $$|S_1| \le 3\mathbf{o}(s)/2-3 \le |R|/2-3=(|R|-6)/2.$$
Since $|S_2|=11$, as long as $|R|\ge 28$ we have $|S_2| \le (|R|-6)/2$. We have this as long as $\mathbf{o}(s)\ge 14$.
\end{proof}

\section{Inductive constructions}\label{sec:induction}
Pivoting on the results in Section~\ref{constructions} we start our main argument. The results in this section are tailored to prove Theorem~\ref{thrm:mainHRR} inductively. Let $R$ be a finite group  and let $N$ be a non-identity proper normal subgroup of $R$. Broadly speaking, Propositions~\ref{prop:cyclic trivial stabiliser} and~\ref{prop:not cyclic trivial stabiliser} allow to construct a HGR for  $R$ using a HGR for $N$. Proposition~\ref{prop:one more time with feeling} serves a similar purpose, in the special case that $N$ is abelian.

\begin{proposition}\label{prop:cyclic trivial stabiliser}
Let $N$ be a non-identity proper normal subgroup of $R$ with $R=\langle N,r\rangle$, for some $r\in R$. Suppose there is some $S_N \subseteq N$ such that $\Aut(\Haar(N,S_N))_0=\hat N$  and $4 \le |S_N| \le (|N|-6)/2$. Define $S_1=S_N \cup (Nr\setminus\{r\})$ and $S_2=R\setminus S_1$.

Then, for $i\in \{1,2\}$, $\hat R=\Aut(\Haar(R,S_i))_0$.  Also, if $g \in \Aut(\Haar(R,S_i))$ is such that $(1,-1)^g =(1,1)$, then $(r,1)^g=(r^{-1},-1)$, $(N\times \{-1,1\})^g=N\times \{-1,1\}$ and the restriction of $g$ to $N\times\{-1,1\}$ is an automorphism of $\mathrm{Haar}(N,S_N)$.

Furthermore, if $|R:N| \ge 3$ then $4 \le |S_1| \le (|R|-6)/2$, while if $|R:N|=2$ then $4 \le |S_2| \le (|R|-6)/2$.
\end{proposition}

\begin{proof}
As  $\Aut(\Haar(N,S_N))_0=\hat N$, by Lemma~\ref{lem:disconnected}, $\Haar(N, S_N)$ is connected and therefore, by Lemma~\ref{connection2}, so is $\Haar(R,S_1)$.\footnote{Observe that, as $|S_N|\le (|N|-6)/2$, we have $|N|\ge 6$ and hence, we are in a position to use Lemma~\ref{lem:disconnected}.}  Using Lemma~\ref{lem:bipartite complement}, this implies $\Aut(\Haar(R,S_1))_0=\Aut(\Haar(R,S_2))_0$ and hence we need only consider one of these except for the calculation of cardinality. We choose to work with $S_1$ and we let 
$A=\mathrm{Aut}(\mathrm{Haar}(R,S_1))$.

Since $\Haar(R,S_1)$ is connected,  $\{R\times \{-1\},R\times \{1\}\}$ is a system of imprimitivity for $A$.

Let $n \in N$. Every element of $(Nr\setminus\{r,nr\})\times\{1\}$ is a common neighbour of $(1,-1)$ and $(n,-1)$. Thus, $(1,-1)$ and $(n,-1)$ have at least $|N|-2$ common neighbours.

Let $t \in R\setminus N$, and consider the common neighbours of $(t,-1)$ and $(1,-1)$. If $t \notin Nr \cup Nr^{-1}$, then $(t,-1)$ has no common neighbours with $(1,-1)$. If $t \in Nr$, then  $(t,-1)$ has at most $|S_N|$ common neighbours with $(1,-1)$\footnote{Namely, the only possible common neighbours are the elements of $S_Nt \times \{1\}$.}, unless $Nr=Nr^{-1}$ in which case the elements of $S_N\times\{1\}$ may also be common neighbours; the same is true (with the sets reversed) if $t \in Nr^{-1}$. This shows that $(t,-1)$ has at most $2|S_N|$ common neighbours with $(1,-1)$. As $$2|S_N|\le 2\frac{|N|-6}{2}=|N|-6,$$ the elements of $R\setminus N$ have at most $|N|-6$ common neighbours with $(1,-1)$.

We have just shown that every element of $N \times \{-1\}$ has more common neighbours with $(1,-1)$ than any element of $R\setminus N \times \{-1\}$. Since the same argument holds for $N\times\{1\}$\footnote{with the role of $(1,-1)$ replaced by $(1,1)$}, $N\times\{-1\}$ and $N\times \{1\}$ are blocks of imprimitivity for $A$.

We now show that $N\times \{-1,1\}$ is also a block of imprimitivity for $A$.\footnote{This is clear when $\mathrm{Aut}(\mathrm{Haar}(R,S))$ fixes the two parts $R\times\{-1\}$ and $R\times\{1\}$ setwise; however, it needs some further motivation when $A$ is transitive, because it is not obvious that an element of $A$ mapping $N\times\{-1\}$ to $N\times\{1\}$ must also map $N\times\{1\}$ to $N\times\{-1\}$.} Let $g\in A$ with $(1,-1)^g\in R\times\{1\}$. Replacing $g$ if necessary, we may suppose that $(1,-1)^g=(1,1)$. We claim that $(N\times\{1\})^g=N\times\{-1\}$ and $(Nr\times\{1\})^g=Nr^{-1}\times\{-1\}$: the first  claim establishes that $N\times\{-1,1\}$ is indeed a block for $A$, and the second claim will be helpful for proving the last assertion in the lemma.
We clearly have $|N|-2 >(|N|-6)/2$. Therefore the elements of $Nr \times \{1\}$ are the only vertices in $\mathrm{Haar}(R,S_1)$ that have $|N|-2$ neighbours in $N\times \{-1\}$. Multiplying by $\rho_{r^{-1}}\in \hat R$, we deduce that the elements of $N\times \{1\}$ are the only vertices in $\mathrm{Haar}(R,S_1)$ that have $|N|-2$ neighbours in $Nr^{-1}\times\{-1\}$. 
Since $g$ is an automorphism, $g$ must preserve this ``uniqueness'' property and hence it must map $N\times\{1\}$ to $N\times\{-1\}$ and $Nr\times\{1\}$ to $Nr^{-1}\times\{-1\}$.

We have shown that $A_{(1,-1)}$ fixes $N\times \{-1,1\}$ setwise, and fixes $(1,-1)$. Thus any element of $A_{(1,-1)}$ induces an automorphism of $\Haar(N,S_N)$ that fixes $(1,-1)$. Since $\hat N=\Aut(\Haar(N,S_N))_0$, this means that $A_{(1,-1)}$ fixes $N \times \{-1, 1\}$ pointwise. 
Similarly, $A_{(1,1)}$ fixes $N \times \{-1,1\}$ setwise, and fixes $(1,1)$; therefore it fixes $N\times 
\{-1,1\}$ pointwise. Every vertex in $Nr^{-1} \times \{1\}$ has a unique non-neighbour in $N \times \{1\}$; since the latter set is fixed pointwise by $A_{(1,1)}$, this implies that $A_{(1,-1)}$ fixes  $Nr^{-1}\times\{1\}$ pointwise. We have now established that $A_{(1,1)}$ fixes $S^{-1} \times \{-1\}$ pointwise. An analogous argument shows that $A_{(1,-1)}$ fixes $S\times\{1\}$ pointwise. 
By Lemma~\ref{lem:done if nbrs fixed}, we deduce $\Aut(\Haar(R,S_1))_0=\hat R$.

Let $g \in A$ with $(1,-1)^g=(1,1)$. From above, $g$ maps $Nr\times\{1\}$ to $Nr^{-1}\times\{-1\}$. In particular, since $(r,1)$ is the unique non-neighbour of $(1,-1)$ in the first set, it must map to $(r^{-1},-1)$, as the unique non-neighbour of $(1,1)$ in the second set.

Suppose that $|R:N| \ge 3$. Then $$|S_1|=|S_N|+|N|-1 \le \frac{3|N|-8}{2}< \frac{|R|-6}{2}.$$ We also have $|S_1| \ge |S_N| \ge 4$.

If $|R:N|=2$ then $|S_1|>|N|>|R|/2$. However, in this case we have $$|S_2|=|R|-|S_N|-|N|+1=|R|/2-|S_N|+1 \le |R|/2-3=(|R|-6)/2$$ since $|S_N| \ge 4$.
\end{proof}

\begin{proposition}\label{prop:not cyclic trivial stabiliser}
Let $N$ be a non-identity proper normal subgroup of $R$ with $R/N$ non-abelian simple and let $r_1,r_2\in R$ with $R=\langle r_1,r_2,N\rangle$\footnote{Recall that every non-abelian simple group can be generated with two elements.} As in Lemma~$\ref{lemma:Rsimple}$, we may assume $\mathbf{o}(r_2N) \ge 5$. Suppose that there is some $S_N \subseteq N$ with $4 \le |S_N| \le (|N|-6)/2$ and  $\Aut(\Haar(N, S_N))_0=\hat N$.

Let $n_1 \in N$ with $n_1 \neq 1$, and define $$S=S_N \cup (Nr_2^{-1}\setminus\{r_2^{-1}\})\cup (Nr_2\setminus\{r_2,n_1r_2\}) \cup \{r_1\} \cup \{r_2r_1\}.$$
Then  $\Aut(\Haar(R,S))=\hat R$, and $4 \le |S| \le (|R|-6)/2$.
\end{proposition}

\begin{proof}
This proof follows a similar strategy to the proof of Proposition~\ref{prop:cyclic trivial stabiliser}. As in that proof, $\Haar(N,S_N)$ must be connected\footnote{By hypothesis $|S_N|\ge 4$. Now, Lemma~\ref{lem:disconnected} implies that $\mathrm{Haar}(N,S_N)$ is connected.}, as is the entire graph $\Haar(R,S)$. So $\{R\times \{-1\},R \times \{1\}\}$ is an invariant partition of $A=\Aut(\Haar(R,S))$.

Since $4 \le |S_N| \le (|N|-6)/2$ we must have $|N|\ge 14$.

For every $r\in R$ and for every $\varepsilon\in \{-1,1\}$,
 any two vertices in $Nr \times \{\varepsilon\}$ have at least $2|N|-6$ common neighbours: there are at least $|N|-2$ common neighbours in  $Nr_2^{\varepsilon} r \times \{-\varepsilon\}$ and at least $|N|-4$ common neighbours in $N r_2^{-\varepsilon} r\times\{-\varepsilon\}$. 

Given $r\in R\setminus N$, we determine an upper bound on the number of common neighbours of $(1,-1)$ and $(r,-1)$. If $Nr \notin \langle Nr_2\rangle$, then in any set $Nr' \times \{1\}$, at least one of $(1,-1)$ and $(r,-1)$ has at most one neighbour in $Nr'\times\{1\}$. Since $(1,-1)$ has neighbours in only five sets of this form, the two vertices certainly have no more than $5<2|N|-6$\footnote{This inequality follows from the fact that $|N|\ge 14$.} common neighbours. We may now assume $Nr \in \langle Nr_2 \rangle$. If $Nr \notin \{r_2^iN: i\in\{-1,1,-2,2\}\}$, then $(1,-1)$ and $(r,-1)$ have at most $2$ common neighbours. If $Nr \in \{Nr_2,Nr_2^{-1}\}$, then since $\mathbf{o}(r_2N)\ge 5$, $(1,-1)$ and $(r,-1)$ have at most $2|S_N|$ common neighbours in $\langle Nr_2 \rangle \times \{1\}$, and possibly up to two more outside this set, for a total of at most $2|S_N|+2<2|N|-6$\footnote{This inequality follows from 
$|S_N|\le (|N|-6)/2$ and $|N|\ge 14$.}. If $rN \in \{Nr_2^{2},Nr_2^{-2}\}$, then since $\mathbf{o}(r_2N)\ge 5$, $(1,-1)$ and $(r,-1)$ have at most $|N|-2$ common neighbours in $\langle r_2N \rangle \times \{1\}$, and possibly up to two more outside this set, for a total of at most $|N|<2|N|-6$.
 
The previous two paragraphs show that  $N \times \{-1\}$ is a block of imprimitivity for $A$. A similar argument implies that $Nr \times \{\varepsilon\}$ is a block of imprimitivity for $A$, for every $r\in R$ and for every $\varepsilon\in \{-1,1\}$.

We can also see that among these blocks, all of the neighbours of $(1,-1)$ lie in $N\times \{1\}$, $Nr_1 \times \{1\}$, $Nr_2 \times \{1\}$, $Nr_2^{-1} \times \{1\}$, and $Nr_2r_1 \times \{1\}$. Furthermore, among these only $N \times \{1\}$ has $|S_N|$ neighbours of $(1,-1)$\footnote{The numbers of neighbours of $(1,-1)$ in the various sets are $|S_N|$, $1$, $|N|-2$, $|N|-1$, and $1$ respectively. As $|N|\ge 14$ and $|S_N|\le (|N|-6)/2$, we have $|S_N|\notin \{|N|-1,|N|-2\}$. By hypothesis, $|S_N|\ge 4$ and hence $|S_N|\ne 1$.} This implies that $N\times \{-1,1\}$ is a block of imprimitivity for $A$.

In particular, $A_{(1,-1)}$ fixes  $N\times\{-1,1\}$ setwise. Since the subgraph induced by $\mathrm{Haar}(R,S)$ on $N\times\{-1,1\}$ is $\Haar(N, S_N)$ and since $\mathrm{Aut}(\mathrm{Haar}(N,S_N))=\hat N$, we deduce that $A_{(1,-1)}$ fixes $N\times\{1,-1\}$ pointwise.
Also, since $Nr_2^{-1} \times \{1\}$ is the unique block that has $|N|-1$ neighbours of $(1,-1)$, it must be fixed setwise by $A_{(1,-1)}$. Since each vertex in $Nr_2^{-1}\times\{1\}$ has a unique  non-neighbour in $N \times \{-1\}$ and every such non-neighbour is fixed by $A_{(1,-1)}$, each vertex of $Nr_2^{-1} \times \{1\}$ is also  fixed by $A_{(1,-1)}$. Since $A_{(1,-1)}$ fixes $Nr_2^{-1}\times\{1\}$ pointwise, conjugating by $(r_2^{-1},r_2^{-1})\in \hat R$, we deduce that $$A_{(1,-1)=}A_{(r_2^{-1},-1)}=A_{(1,-1)}^{(r_2^{-1},r_2^{-1})}$$ fixes $(Nr_2^{-1}\times\{1\})^{(r_2^{-1},r_2^{-1})}=Nr_2^{-2}\times\{1\}$ pointwise. An inductive argument yields that $A_{(1,-1)}$ fixes $Nr_2^j\times\{-1\}$ pointwise, for every $j$.

We can make an exactly parallel argument to show that $A_{(1,1)}$ fixes $(N \times \{1,-1\})$ and $Nr_2^j\times \{-1\}$ pointwise, for every $j$.

Since $Nr_1 \times \{1,-1\}$ and $Nr_2r_1 \times \{1,-1\}$ are blocks for $A$ and contain the remaining neighbours $(r_1,1)$ and $(r_2r_1,1)$ of $(1,-1)$, $A_{(1,-1)}$ must either fix each setwise, or interchange them. Suppose that there exists $g\in A_{(1,-1)}$ with $$(Nr_1\times\{-1\})^g=Nr_2r_1\times\{-1\}\hbox{ and }(Nr_2r_1\times\{-1\})^g=Nr_1\times\{-1\}.$$ In particular, $(r_1,1)^g=(r_2r_1,1)$ and $(r_2r_1,1)^g=(r_1,1)$. Since $Nr_1\times\{-1,1\}$ and $Nr_2r_1\times\{-1,1\}$ are  blocks of imprimitivity, we have 
$(Nr_1\times\{1\})^g=Nr_2r_1\times\{1\}$ and $(Nr_2r_1\times\{1\})^g=Nr_1\times\{1\}$.
However, these vertices can be distinguished because $(r_1,1)$ has $|N|-1$ neighbours in $Nr_2r_1 \times \{-1\}$, while $(r_2r_1,1)$ has only $|N|-2$ neighbours in $Nr_1 \times \{-1\}$. Therefore, we cannot have $(r_1,1)^g=(r_2r_1,1)$ and $(Nr_2r_1\times\{-1\})^g=Nr_1\times\{-1\}$. This contradiction has shown that $A_{(1,-1)}$ fixes each of $Nr_1\times\{1\}$ and $Nr_2r_1\times\{1\}$ setwise.
Therefore, $(r_1,1)$ and $(r_2r_1,1)$ must be fixed by $A_{(1,-1)}$.  Summing up, we have shown that $A_{(1,-1)}$ fixes $S\times\{1\}$ pointwise.

We must use a different trick to distinguish $(r_1^{-1},-1)$ from $(r_1^{-1}r_2^{-1},-1)$ under the action of $A_{(1,1)}$. Each has exactly two neighbours in $\langle Nr_2 \rangle \times \{1\}$, one of which is $(1,1)$. For $(r_1^{-1},-1)$, the other is $(r_2,1)$. For $(r_1^{-1}r_2^{-1},-1)$, the other neighbour is $(r_2^{-1},-1)$. As we have already shown that $A_{(1,1)}$ fixes $(r_2,-1)$ and $(r_2^{-1},-1)$, $A_{(1,1)}$ must also fix $(r_1^{-1},-1)$ and $(r_1^{-1}r_2^{-1},-1)$. We have now shown that $A_{(1,1)}$ fixes $S^{-1} \times \{-1\}$ pointwise. Therefore Lemma~\ref{lem:done if nbrs fixed} yields $\Aut(\Haar(R,S_1))_0=\hat R$.

Let $g\in A\setminus \hat R$. Then $(1,-1)^g\in R\times\{1\}$. Replacing $g$ if necessary, we may suppose that $(1,-1)^g=(1,1)$. As $N\times\{-1,1\}$ is a block of imprimitivity for $A$ and as $g$ maps a vertex in $N\times\{-1,1\}$ to another vertex in $N\times\{-1,1\}$, we deduce that $g$ fixes $N\times\{-1,1\}$ setwise.   Arguments made above using the numbers of neighbours between blocks imply that any element of $A$ that fixes $N \times\{1,-1\}$ setwise, must fix $\langle Nr_2 \rangle \times \{1,-1\}$ setwise. Since $(1,-1)$ has exactly two neighbours outside this set: namely, $(r_1,1)$ and $(r_2r_1,1)$, these must be mapped by $g$ to the only two neighbours of $(1,1)$ outside $\langle Nr_2 \rangle \times \{1,-1\}$: namely, $(r_1^{-1},-1)$ and $(r_1^{-1}r_2^{-1},-1)$. Thus
$$\{(r_1,1),(r_2r_1,1)\}^g=\{(r_1^{-1},-1),(r_1^{-1}r_2^{-1},-1)\}.$$We have also just observed in the preceding paragraph that each of $(r_1^{-1},-1)$ and $(r_1^{-1}r_2^{-1},-1)$ has two neighbours in $\langle Nr_2 \rangle \times \{1,-1\}$. Therefore each of $(r_1,1)$ and $(r_2r_1,1)$
 must also have two neighbours in $\langle Nr_2\rangle \times \{1,-1\}$. All but two of the neighbours of $(r_2r_1,1)$ have a first coordinate that lies in $r_2r_1(r_2^iN)^{-1}=r_2^{1-i}r_1N$ for some $i$; none of these can be in $\langle Nr_2\rangle$. The remaining neighbours of $(r_2r_1,1)$ are $(1,-1)$, and $(r_1^{-1}r_2r_1,-1)$. So  $r_1^{-1}r_2r_1$ belongs to the set $\langle Nr_2\rangle$. In other words, modulo $N$, $r_1$ normalises $\langle r_2\rangle$; however this contradicts $R/N$ being non-abelian simple.
 
We conclude by bounding the cardinality of $|S|$. Since $4 \le |S_N|$ and $S_N \subset S$, we certainly have $|S|\ge 4$. Since $R/N$ is nonabelian simple we have $|R:N| \ge 60$, so $|S|=|S_N|+|N|-1+|N|-2+2 < 3|N|-1< (|R|-6)/2$.
\end{proof}

\begin{corollary}\label{cor:HRR-induct}
Let $N$ be a non-identity proper normal subgroup of $R$ with  $R/N$ cyclic or  non-abelian simple. Suppose that there is some $S_N \subseteq N$ with $4 \le |S_N| \le (|N|-6)/2$, such that $\Aut(\Haar(N,S_N))_0=\hat N$. Then there is some $S_R \subseteq R$ with $4\le |S_R| \le (|R|-6)/2$ such that $\Aut(\Haar(R,S_R))_0=\hat R$.

Furthermore, if $\Aut(\Haar(N,S_N))=\hat N$ or if $R/N$ is non-abelian simple, then $\Aut(\Haar(R,S_R))=\hat R$. 
\end{corollary}

\begin{proof}
If $R/N$ is cyclic, then take $S_R=S_i$ from Proposition~\ref{prop:cyclic trivial stabiliser}, with $i=1$ if $|R:N|>2$ and $i=2$ if $|R:N|=2$. If $R/N$ is non-abelian simple, then take $S_R=S$ from Proposition~\ref{prop:not cyclic trivial stabiliser}.

Taking $A=\Aut(\Haar(R,S_R))$, Proposition~\ref{prop:cyclic trivial stabiliser} or~\ref{prop:not cyclic trivial stabiliser} gives $A_0=\hat R$. Furthermore, if $A>\hat R$, then $R/N$ is cyclic and there is some $g \in A$ such that $(1-1)^g =(1,1)$, and $g$ induces an automorphism of the induced subgraph on $N\times\{-1,1\}$ (which is $\Haar(N,S_N)$). This is not possible if $\Aut(\Haar(N,S_N))=\hat N$, so in this case $A=\hat R$.
\end{proof}

Before proving Proposition~\ref{prop:one more time with feeling}, which constitutes our last inductive tool, we need an auxiliary technical remark.
\begin{lemma}\label{moura}Let $R$ be a non-abelian finite group, let $N$ be an abelian normal subgroup of $R$ with $|R:N|=2$ and let $r\in R\setminus N$. Then one of the following holds
\begin{enumerate}
\item\label{moura1}there exists a proper subgroup $N_1$ of $N$, $n_2\in N\setminus N_1$ and $n_1\in N_1\setminus\{1\}$ with $N=\langle N_1,n_2\rangle$ and $ rn_2r^{-1}n_2^{-1}\notin \{1,n_1\}$,
\item\label{moura2}$R$ is a dihedral group of order $2p$, for some prime $p$, or
\item\label{moura3}$R$ is the dihedral group of order $8$ or the quaternion group of order $8$.
\end{enumerate}
\end{lemma}
\begin{proof}
Let $C={\bf C}_N(r)$. As $R=\langle N,r\rangle$ is non-abelian and as $N$ is abelian, $C$ is a proper subgroup of $N$. Let $N_1$ be a maximal subgroup of $N$ containing $C$. In particular, $|N:N_1|=p$, for some prime number $p$. Also, $C \le N_1$ implies $1 \neq rn_2r^{-1}n_2^{-1}$.

If $N_1=1$, then $N$ is a cyclic group of order $p$ and hence $R$ is a dihedral group, thus part~\eqref{moura2} is satisfied. Therefore, suppose $N_1\ne 1$.

Let $n_2\in N\setminus N_1$. Since $N_1$ is maximal in $N$, we have $N=\langle N_1,n_2\rangle$. If $|N_1|\ge 3$, then there exists $n_1\in N_1$ with $n_1\ne 1$ and $n_1\ne rn_2r^{-1}n_2^{-1}$, thus part~\eqref{moura1} is satisfied. Therefore, suppose $|N_1|=2$. As $C\le N_1$, we deduce $N_1=C={\bf C}_N(r)$.

Assume $p$ odd and let $P=\langle x\rangle$ be a Sylow $p$-subgroup of $N$. Clearly, $N=\langle N_1,x\rangle$. Moreover, $rxr^{-1}=x^{-1}$ and hence $rxr^{-1}x^{-1}=x^{-2}\notin N_1$. Therefore, part~\eqref{moura1} is satisfied by taking $n_2=x$ and $n_1$ the non-identity element of $N_1$.

Finally, assume $p=2$. Thus $|N|=4$ and hence $|R|=8$. Since $R$ is non-abelian, we deduce that part~\eqref{moura3} is satisfied.
\end{proof}

\begin{proposition}\label{prop:one more time with feeling}
Let $R$ be a non-abelian group, let $N$ be an abelian subgroup of $R$ with $|R:N|=2$ and let $r\in R\setminus N$. Suppose that $R$ is neither a dihedral group of order $8$, nor of order $2p$ for some prime $p$, nor a quaternion group. Let $N_1,n_1$ and $n_2$ be as in Lemma~$\ref{moura}$ part~$\eqref{moura1}$.
 Suppose  there exists $S_{N_1} \subseteq N_1$ with $4 \le |S_{N_1}| \le (|N_1|-6)/2$ and $\Aut(\Haar(N_1, S_{N_1}))_0=\hat{N_1}$.

Let $$S=S_{N_1} \cup (n_2N_1\setminus\{n_2\})\cup \{r,n_2r,n_1n_2r\}.$$
Then $\Aut(\Haar(R,S))=\hat R$, and $4 \le |S| \le (|R|-6)/2$. 
\end{proposition}

\begin{proof}
Again we follow the same template as in the proof of Proposition~\ref{prop:cyclic trivial stabiliser}. As in that proof, $\Haar(N_1,S_{N_1})$ is connected, as is the entire graph $\Haar(R,S)$. So $\{R\times \{-1\},R \times \{1\}\}$ is a system of imprimitivity for $A=\Aut(\Haar(R,S))$.

For every $r'\in R$ and for every $\varepsilon\in\{-1,1\}$, as in the  proof of Propositions~\ref{prop:cyclic trivial stabiliser} and~\ref{prop:not cyclic trivial stabiliser}, any two vertices in the same set $N_1r' \times \{\varepsilon\}$\footnote{Observe that $N_1$ may not be normal in $R$, so we must be particularly careful about left and right cosets. However, $N_1$ is normal in $N$, which will be useful.} and with the same second coordinate have many common neighbours, in this case at least the $|N_1|-2$ common neighbours that lie in $N_1n_2^{-\varepsilon}r' \times \{-\varepsilon\}$.

Let $x\in R\setminus N_1$. We count the common neighbours of $(1,\varepsilon)$ and $(x,\varepsilon)$.  If $x \notin N$, then since any vertex in $Nr' \times \{\varepsilon\}$ has at most $3$ neighbours outside $Nr' \times \{-\varepsilon\}$, $(1,\varepsilon)$ and $(x,\varepsilon)$ have at most $6$ common neighbours\footnote{at most $3$ in $N \times \{-\varepsilon\}$ and at most $3$ in $Nx \times \{-\varepsilon\}$}. If $x\in N\setminus N_1$, then $(x,\varepsilon)$ and $(1,\varepsilon)$ have at most $2|S_{N_1}|+3$ common neighbours\footnote{at most $|S_{N_1}|$ in each of $N_1 \times \{-\varepsilon\}$ and $N_1x \times \{-\varepsilon\}$, and at most $3$ common neighbours outside $N \times \{-\varepsilon\}$}. 

Since $|S_{N_1}|\ge 4$, $2|S_{N_1}|+3$ is an upper bound on the number of common neighbours of two vertices that are not in the same set $N_1r' \times \{\varepsilon\}$, for some $r'\in R\setminus N_1$ and some $\varepsilon \in \{-1,1\}$.

Since $|S_{N_1}| \le (|N_1|-6)/2$, we have $$2|S_{N_1}|+3 < |N_1|-2,$$ so each set $N_1r'\times \{\varepsilon\}$  is a block for $A$,  for every $r' \in R$ and for every $\varepsilon \in \{-1,1\}$. 

We  see that  almost all of the neighbours of any vertex in $N_1r' \times \{-1\}$ lie in $Nr'\times \{1\}$, with a total of $3$ others appearing in the blocks that lie inside $Nrr' \times \{1\}$ and $Nr^{-1}r' \times \{1\}$. In fact, we can say more about these three other neighbours. Suppose $(nr',-1) \in N_1r' \times \{-1\}$. Then the three ``other" neighbours of this vertex are $(rnr',1)$, $(n_2rnr',1)$, and $(n_1n_2rnr',1)$. If we choose $r'' \in R$ so that $rnr' \in N_1r''$, then $(rnr',1) \in  N_1r''\times \{1\}$, but  $(n_2rnr',1),(n_1n_2rnr',1) \in n_2N_1r''\times\{1\}$, and $n_2N_1r''=N_1n_2r'' \neq N_1r''$. Therefore, any vertex in   $N_1r' \times \{-1\}$ has $|S_{N_1}|$ neighbours in $N_1r'\times \{1\}$, $|N_1|-1$ neighbours in $N_1n_2r' \times \{1\}$, $2$ neighbours in $N_1n_2r'' \times \{1\}$, and a final neighbour in $N_1r''\times \{1\}$. 
By  our hypothesis on $|S_{N_1}|$, none of the other values is equal to $|S_N|$. This implies that $N_1r'\times \{-1,1\}$ is a block of $A$.

As in previous proofs, if  $g \in A_{(1,-1)}$ then $g$ fixes $(1,-1)$ and fixes $N_1 \times \{-1,1\}$ setwise, so since the induced subgraph on $N_1\times \{-1,1\}$ is $\Haar(N_1, S_{N_1})$, $g$ must fix $N_1 \times \{-1,1\}$ pointwise. Since we showed in the previous paragraph that $(r,1)$ is the only neighbour of $(1,-1)$ that is a unique neighbour of $(1,-1)$ in some right coset of $N_1$, we must have $(r,1)^g=(r,1)$. Similarly, $N_1 \times \{-1,1\}$ and $(r^{-1},-1)$ are fixed pointwise by $A_{(1,1)}$. 

At this point, looking back through everything we have determined so far in this proof, we observe that if $S'=S\setminus\{n_2r,n_1n_2r\})$ and $B=\Aut(\Haar(R,S')$ then $B$ has the same blocks we have found for $A$. Also, $B_{(1,-1)}$ fixes $S'\times \{1\}$ pointwise, and $B_{(1,1)}$ fixes $S^{-1} \times \{-1\}$ pointwise. Therefore by Lemma~\ref{lem:done if nbrs fixed}, we have $B_0=\hat R$. Since we have shown that in $\Haar(R,S)$ we can distinguish edges that arise from $\{n_2r,n_1n_2r\}$ from all of the other edges, we must have $A \le B$.
Therefore we must have $A_0=\hat R$. 

We still need to show that  $A=A_0$. 

Suppose that there is some $\varphi \in A$ such that $(1,-1)^\varphi=(1,1)$. For each $\varepsilon \in \{-1,1\}$,  $N_1 \times \{\varepsilon\}$ is a block of $A$, and $N_1 \times \{-1,1\}$ is a block of $A$. Therefore $\varphi$ induces an automorphism of $\Haar(N_1, S_{N_1})$. Since $N_1$ is abelian and $\Aut(\Haar(N_1,S_{N_1}))_0=\hat{N_1}$, the only possibility is that $\hat \varphi$ from Lemma~\ref{group automorphism} must act as inversion on $N_1$. Also $(1,1)^\varphi=(1,-1)$. Now, $N_1r^{-1}\times \{-1\}$  is the only block that has exactly one neighbour of $(1,1)$, and that neighbour is $(r^{-1},-1)$. Similarly, $N_1r\times \{1\}$ is the only block that has exactly one neighbour of $(1,-1)$, and that neighbour is $(r,1)$. Therefore $(r^{-1},-1)^\varphi=(r,1)$, and accordingly $(r^{-1})^{\hat\varphi}=r$. Finally, the only two remaining neighbours of $(1,1)$ must map to the only two remaining neighbours of $(1,-1)$, so $\{(r^{-1}n_2^{-1},-1),(r^{-1}n_2^{-1}n_1^{-1},-1)\}^\varphi=\{(n_2r,1),(n_1n_2r,1)\}$. 

 If $(r^{-1}n_2^{-1})^{\hat\varphi}=n_2r$ then since $(r^{-1}n_2^{-1})^{\hat\varphi}=(r^{-1})^{\hat\varphi}(n_2^{-1})^{\hat \varphi}=rn_2$, this contradicts our hypothesis that $1 \neq rn_2r^{-1}n_2^{-1}$. So we must have $(r^{-1}n_2^{-1})^{\hat\varphi}=n_1n_2r$. Again, this must be equal to $rn_2$. Therefore $n_1=rn_2r^{-1}n_2^{-1}$. But this contradicts our hypothesis on $n_1$. So there is no such $\varphi$, and we must have $A=A_0$.
 
Since $4 \le |S_{N_1}|$ and $S_{N_1} \subset S$, we have $|S|\ge 4$. Since $N_1 <N<R$ we have $|R:N_1| \ge 4$, so $$|S|=|S_{N_1}|+|N_1|+2 <\frac{|N_1|-6}{2}+|N_1|+2=\frac{3|N_1|}{2}-1 \le \frac{4|N_1|}{2}-8<\frac{|R|-6}{2},$$ where we have deduced $|N_1| \ge 14$ from $4 \le |S_{N_1}|\le (|N_1|-6)/2$ and have used this in our calculations.
\end{proof}

\section{Proof of Theorem~$\ref{thrm:mainHRR}$}
In this section we prove Theorem~$\ref{thrm:mainHRR}$. We actually prove a slightly stronger form, which will be useful for applying the results from Section~\ref{sec:induction}.
\begin{theorem}\label{thrm:mainHRRRR}
Let $R$ be a finite group. Then one of the following holds
\begin{enumerate}
\item\label{thrm:mainHRRR11}there exists a subset $S$ of $R$ with $4\le|S|\le (|R|-6)/2$ such that $\mathrm{Haar}(R,S)$  is a $\mathrm{HGR}$, 
\item\label{thrm:mainHRRR22}$R$ is an abelian group and there exists $S\subseteq R$ with $4\le |S|\le (|R|-6)/2$ and $\mathrm{Aut}(\mathrm{Haar}(R,S))=\hat R\rtimes\langle\iota\rangle$, where $\iota$ is defined in~$\eqref{iota}$,
\item\label{thrm:mainHRRR33}$R$ is one of the ten exceptional groups appearing in Table~$\ref{auxiliary}$,
\item\label{thrm:mainHRRR44}$|R|\le 13$.
\end{enumerate}
\end{theorem}
\begin{proof}
We argue by induction on $|R|$. When $|R|\le 13$, part~\eqref{thrm:mainHRRR44} is satisfied. Suppose  $|R|\ge 14$ and let $N$ be a normal subgroup of $R$ with $R/N$ simple. In particular, $R/N$ is either cyclic of prime order, or a non-abelian simple group. As $|N|<|R|$, we may assume that the result holds for $N$.

\smallskip

\noindent\textsc{Assume that $N$ satisfies part~\eqref{thrm:mainHRRR11}.}

\smallskip

\noindent Then Corollary~\ref{cor:HRR-induct} immediately implies that $R$ satisfies part~\eqref{thrm:mainHRR11}.

\smallskip

\noindent\textsc{Assume that $N$ satisfies part~\eqref{thrm:mainHRRR22}.}

\smallskip

\noindent Suppose $R/N$ is cyclic of order $p$ and let $r\in R$ with $R=\langle N,r\rangle$. 
Here we apply Proposition~\ref{prop:cyclic trivial stabiliser}. When $p\ge 3$, we let $S_R=S_1$ and, when $p=2$, we let $S_R=S_2$, where $S_1$ and $S_2$ are defined in Proposition~\ref{prop:cyclic trivial stabiliser}. We have $4\le |S_R|\le (|R|-6)/2$ and $\mathrm{Aut}(\mathrm{Haar}(R,S_R))_0=\hat R$. If $\mathrm{Aut}(\mathrm{Haar}(R,S_R))=\hat R$, then $R$ satisfies part~\eqref{thrm:mainHRRR11} and hence we may suppose $\hat R<\mathrm{Aut}(\mathrm{Haar}(R,S_R))$. From Proposition~\ref{prop:cyclic trivial stabiliser}, there exists $g\in \mathrm{Aut}(\mathrm{Haar}(R,S))$ with $(1,-1)^g=(1,1)$, $(r,1)^g=(r^{-1},-1)$, $(N\times\{-1,1\})^g=N\times\{-1,1\}$ and with $g$ restricting to an automorphism of the Haar graph $\mathrm{Haar}(N,S_N)$. Since $N$ satisfies part~\eqref{thrm:mainHRRR22}, we have $(n,\varepsilon)^g=(n^{-1},-\varepsilon)$, $\forall (n,\varepsilon)\in N\times\{-1,1\}$. Using Lemma~\ref{group automorphism}, let $\hat g\in \mathrm{Aut}(R)$ and $x\in R$ with $(y,-1)^g=(y^{\hat g},1)$ and $(y,1)^g=(y^{{\hat g}^{-1}}x,-1)$, $\forall y\in R$. Since $(x,-1)=(1,1)^g=(1,-1)$, we deduce $x=1$. For every $n\in N$, we have $r^{-1}nr\in N$ and hence $(r^{-1}nr)^{\hat g}=r^{-1}n^{-1}r$, but we also have
$$(r^{-1}nr)^{\hat g}=(r^{\hat g})^{-1}n^{\hat g}r^{\hat g}=rn^{-1}r^{-1}. $$
This implies $$r^{-1}n^{-1}r=rn^{-1}r^{-1},$$
that is, $r^2$ centralizes $N$. When $|R:N|=p$ is odd, this implies that $r$ centralises $N$ and hence $R$ is abelian and part~\eqref{thrm:mainHRRR22} immediately follows. The same conclusion holds when $p=2$ and $R$ is abelian, because in this case $g$ is the automorphism $\iota$ defined in~\eqref{iota}. Therefore we may suppose that $p=2$ and $R$ is not abelian. Now, Proposition~\ref{prop:one more time with feeling} implies that part~\eqref{thrm:mainHRRR11} holds, except possibly when $R$ is a dihedral group of order $2p$. This case is covered by Proposition~\ref{prop:2gen} and again implies that part~\eqref{thrm:mainHRRR11} holds.

Assume finally that $R/N$ is a non-abelian simple group. Here, Proposition~\ref{prop:not cyclic trivial stabiliser} implies that part~\eqref{thrm:mainHRRR11} is satisfied.

In particular, for the rest of the argument we may suppose that 
\begin{center}
$(\dag):$  every proper normal subgroup of $R$ satisfies part~\eqref{thrm:mainHRRR33} or~\eqref{thrm:mainHRRR44}.
\end{center}

\smallskip

\noindent\textsc{Suppose that $R/N$ is cyclic.}

\smallskip

\noindent  Let $p=|R:N|$. If $p<31$, then $|R|=p|N|\le 29\cdot 2^5=928$ and the veracity of this result can be checked with a computer using the library of small groups in magma.\footnote{We did find a number of additional theoretical constructions that deal with many of these cases, but the additional constructions use similar ideas to the constructions already presented in this paper, and we did not believe they contained sufficient value to include them for results that are computationally straightforward.} Assume $p>31$ and let $P$ be a Sylow $p$-subgroup of $R$. When $|N|\le 13$ or when $N$ is in Table~\ref{auxiliary}, we see that $p$ is relatively prime to $N$ and hence $R=N\rtimes P$. 
Moreover, no group of order at most $13$ and no group in Table~\ref{auxiliaryy} admits an automorphism of order $p$ and hence $R=N\times P$. As $P\unlhd R$, this contradicts $(\dag)$, except when $R=P$. However, when $R=P$, Lemma~\ref{cyclic2} implies that $R$ satisfies part~\eqref{thrm:mainHRRR22}.

\smallskip

\noindent\textsc{Suppose that $R/N$ is non-abelian simple.}

\smallskip

\noindent Let $C={\bf C}_R(N)$ and observe $C\unlhd R$. Assume first $R=C$. Then $R={\bf C}_R(N)$ and hence $N={\bf Z}(R)$. If $R'<R$, then we contradict $(\dag)$. Therefore $R=R'$ and $R$ is quasi-simple.\footnote{Recall that a group $R$ is said to be quasi-simple if $R$ equals its derived subgroup $R'$ and $R/{\bf Z}(R)$ is non-abelian simple, where we are denoting with ${\bf Z}(R)$ the centre of $R$.}
Suppose that $R/{\bf Z}(G)$ is not isomorphic to one of the exceptional cases listed in Lemma~\ref{lemma:Rsimple}. Then, from Lemma~\ref{lemma:Rsimple}, there exist $s,t\in R$ with $R=\langle s,t\rangle {\bf Z}(R)$  and ${\bf o}(s)\ge 12$. We have
$$R=R'=(\langle s,t\rangle {\bf Z}(R))'=\langle s,t\rangle'{\bf Z}(R)'=\langle s,t\rangle'$$
and hence $R=\langle s,t\rangle$.
Set $S=\{s^i:0\le i\le 6\}\cup\{s^8,t,ts\}$. From Proposition~\ref{prop:2gen}, $\mathrm{Haar}(R,S)$ is a HGR, and since $|R:\langle s \rangle|>2$, we also have $4 \le |S| \le (|R|-6)/2$. When $R/{\bf Z}(R)$ is one of the exceptional cases listed in Lemma~\ref{lemma:Rsimple}, the proof follows  with the help of a computer. Therefore, we may suppose that $C<R$ and hence, by $(\dag)$, $C$ has order at most $13$ or is isomorphic to a group in Table~\ref{auxiliary}.

 Since $R/N$ is non-abelian simple, we  deduce $C\le N$. The action of $R$ by conjugation on $N$ gives rise to a group homomorphism from $R$ to $\mathrm{Aut}(N)$. The kernel of this homomorphism is ${\bf C}_R(N)=C$. As $C\le N$, $R/C$ is not solvable and hence $N$ is a group in Table~\ref{auxiliary} having a non-solvable automorphism group. A direct inspection of Table~\ref{auxiliary} shows that $N$ is isomorphic to one of the following groups
 $$C_2^3,\,C_2^4,\,C_2^5.$$

 When $N=C_2^3$, we have $\mathrm{Aut}(N)\cong\mathrm{SL}_3(2)$. Since no proper subgroup of $\mathrm{SL}_3(2)$ is solvable, we deduce $R/N\cong\mathrm{SL}_3(2)$. In particular, $R$ is a group extension of $C_2^3$ via $\mathrm{SL}_3(2)$. Therefore, we may check the veracity of the statement with a computer computation, constructing all group extensions of $C_2^3$ via $\mathrm{SL}_3(2)$. The analysis in all other cases is similar and is performed with a computer.
\end{proof}


\thebibliography{10}

\bibitem{BabaiI}
L. Babai, W. Imrich, Tournaments with given regular group,
\textit{Aequationes Math.} \textbf{19} (1979), 232--244.


\bibitem{Babai}
L.~Babai, Finite digraphs with given regular automorphism groups,
\textit{Period. Math. Hungar.} \textbf{11} (1980), 257--270.


\bibitem{BaGo}L.~Babai, C.~D.~Godsil, On the automorphism groups of almost all Cayley graphs, \textit{European J. Combin.} \textbf{3} (1982), 9--15.

\bibitem{Barmak}J.~A.~Barmak, Regular and semi-regular representations of groups by posets, \href{ 	
https://doi.org/10.48550/arXiv.2307.03106}{arXiv:10.48550/arXiv.2307.03106}.

\bibitem{birkhoff}G.~Birkhoff, Sobre los grupos de automorfismos (On groups of automorphisms), \textit{Rev. Un. Mat. Argentina} \textbf{11} (1946), 155--157.


\bibitem{magma}
W. Bosma, C. Cannon, C. Playoust, The MAGMA algebra system I: The user language,
\textit{J. Symbolic Comput.} \textbf{24} (1997), 235--265.



\bibitem{atlas} J.~H.~Conway, R.~T. Curtis, S.~P.~Norton, R.~A.~Parker, R.~A.~Wilson, An $\mathbb{ATLAS}$ of Finite Groups \textit{Clarendon Press, Oxford}, 1985; reprinted with corrections 2003.

\bibitem{DFS}
J.~-L.~Du, Y.~-Q.~Feng, P.~Spiga, A classification of the graphical $m$-semiregular representations of finite groups, \textit{J. Combin. Theory Ser. A} \textbf{171} (2020), 105174.

\bibitem{DFS1}
J.~-L.~Du, Y.~-Q.~Feng, P.~Spiga, On the existence and the enumeration of bipartite regular representations of Cayley graphs over abelian groups, \textit{J. Graph Theory} \textbf{95} (2020), 677--701.

\bibitem{DJSalso}
J.-L.~Du, Y.-Q.~Feng, P.~Spiga, 
On Haar digraphical representations of groups,
\textit{Discrete Math.} \textbf{343} (2020), 112032.

\bibitem{DJS}J.-L.~Du, Y.-Q.~Feng, P.~Spiga,  On $n$-partite digraphical representations of finite groups, \textit{J. Combin. Theory Ser. A} \textbf{189} (2022), Paper No. 105606, 13 pp. 


\bibitem{EP}I.~Est\'elyi, T.~Pisanski, 
Which Haar graphs are Cayley graphs? \textit{Electron. J. Combin.} \textbf{23} (2016), no. 3, Paper 3.10, 13 pp. 

\bibitem{FKWY}Y.~-Q.~Feng, I.~Kov\'acs, J.~Wang, D.~-W.~Yang, 
Existence of non-Cayley Haar graphs.
\textit{European J. Combin.} \textbf{89} (2020), 103146, 12 pp.

\bibitem{FKY}Y.~-Q.~Feng, I.~Kov\'acs, D.~-D.~Yang, 
On groups all of whose Haar graphs are Cayley graphs,
\textit{J. Algebraic Combin.} \textbf{52} (2020), no. 1, 59--76. 

\bibitem{Gan}Y.~Gan, P.~Spiga, B.~Xia, Asymptotic enumeration of Haar graphical representations, in preparation.

\bibitem{Godsil}C.~D.~Godsil,
GRRs for nonsolvable groups,
\textit{Colloq. Math. Soc. J\'{a}now Bolyai} \textbf{25} (1983), 221--239.

\bibitem{GuralnickKantor}R.~M.~Guralnick, W.~M.~Kantor, Probabilistic generation of finite groups, \textit{J. Algebra} \textbf{234} (2000), 743--792.


\bibitem{KantorSeress}W.~M.~Kantor, \`A.~Seress, Large element orders and the characteristic of Lie-type simple groups,
\textit{J. Algebra} \textbf{322} (2009),  802--832.


\bibitem{MorrisSpiga} J.~Morris, P.~Spiga, Asymptotic enumeration of Cayley digraphs, \textit{Israel J. Math.} \textbf{242} (2021), 401--459.


\bibitem{MorrisSpiga1}
J. Morris, P. Spiga, Every finite non-solvable group admits an oriented regular representation, \textit{J. Combin. Theory Ser. B} \textbf{126} (2017), 198--234.

\bibitem{MorrisSpiga3}J. Morris, P.~Spiga, Classification of finite groups that admit an oriented regular representation, \textit{Bull. Lond. Math. Soc.} \textbf{50} (2018), 811--831.

\bibitem{Spiga}
P. Spiga, Finite groups admitting an oriented regular representation,
\textit{J. Combin. Theory Ser. A} \textbf{153} (2018), 76--97.


\bibitem{DFR}P.~Spiga, On the existence of Frobenius digraphical representations, \textit{The Electronic Journal of Combinatorics} \textbf{25} (2018), paper \#P2.6.
\bibitem{GFR}P. Spiga, On the existence of graphical Frobenius representations and their asymptotic enumeration: an answer to the GFR conjecture, \textit{J. Comb. Theory Series B}, \textbf{142} (2020), 210--243.

\bibitem{spiga1}P.~Spiga, The maximum order of the elements of a finite symplectic group of even characteristic,
\textit{Comm. Algebra} \textbf{43} (2015), 1417--1434.

\bibitem{Xiaf}
B.~Z.~Xia, T.~Fang, Cubic graphical regular representations of $\mathrm{PSL}_2(q)$, \textit{Discrete Math.} \textbf{339} (2016), 2051--2055.

\bibitem{XiaZheng}B.~Xia, S.~Zheng, Asymptotic enumeration of graphical regular representations, \textit{Proc. London Math.
Soc.} \textbf{127} (2023), 1424--1450.
\end{document}